\documentclass[reqno]{amsart}

\usepackage{amsmath, amssymb, amsfonts, amsthm}
\usepackage{hyperref}
\usepackage{graphicx}

\theoremstyle{plain}
\newtheorem{theorem}{Theorem}[section] 
\newtheorem{corollary}[theorem]{Corollary}
\newtheorem{lemma}[theorem]{Lemma}
\newtheorem{proposition}[theorem]{Proposition}
\newtheorem{definition}[theorem]{Definition}

\theoremstyle{remark}
\newtheorem{remark}[theorem]{Remark}

\usepackage{bm}

\def\Re{\mathop{\mathrm{Re}}\nolimits}
\def\Im{\mathop{\mathrm{Im}}\nolimits}

\def\Z{\mathrm{Z}}
\def\Zm{\mathrm{Z}^{(m)}}
\def\Od{\Omega^\delta}
\def\contOd{\widehat\Omega^\delta}
\def\Int{\mathrm{Int}}
\def\const{\mathrm{const}}
\def\crad{\mathrm{crad}}
\def\dist{\mathrm{dist}}

\begin{document}

\title[On the convergence of massive LERW to massive SLE(2) curves]{On the convergence of massive loop-erased random walks to massive SLE(2) curves}

\author[Dmitry Chelkak]{Dmitry Chelkak$^\mathrm{a,b}$}

\author[Yijun Wan]{Yijun Wan$^\mathrm{a}$}

\thanks{\textsc{${}^\mathrm{A}$ D\'epartement de math\'ematiques et applications de l'ENS, \'Ecole Normale Sup\'erieure PSL Research University, CNRS UMR 8553, 45 rue d'Ulm, Paris, France.} }

\thanks{\textsc{${}^\mathrm{B}$ Holder of the ENS--MHI chair funded by MHI. On leave from St.~Petersburg Dept. of Steklov Mathematical Institute RAS, Fontanka 27, St.~Petersburg, Russia.}}

\thanks{\emph{E-mail:} \texttt{dmitry.chelkak@ens.fr}, \texttt{yijun.wan@ens.fr}}

\begin{abstract}
Following the strategy proposed by Makarov and Smirnov~\cite{Makarov-Smirnov} in 2009 (see also~\cite{BBK,BBC} for theoretical physics arguments), we provide technical details for the proof of convergence of massive loop-erased random walks to the chordal mSLE(2) process. As no follow-up to~\cite{Makarov-Smirnov} appeared since then, we believe that such a treatment might be of interest to the community. We do not require any regularity of the limiting planar domain  near its degenerate
prime ends $a$ and $b$ except that $(\Omega^{\delta},a^{\delta},b^{\delta})$ are assumed to be ‘close discrete approximations’ to $(\Omega, a, b)$ near $a$ and $b$ in the sense of a recent work \cite{karrila}.
\end{abstract}



\maketitle

\section{Introduction}
The classical loop-erased random walk (LERW) in a discrete domain~$\Od\subset\delta\mathbb Z^2$ is a curve obtained from a simple random walk trajectory by erasing the loops in chronological order. In the famous paper~\cite{LSW-LERW} the convergence of such trajectories to the so-called SLE(2) curves (see~\cite{lawler-book,kemppainen-book,berestycki-norris} and references therein) was proved by Lawler, Schramm and Werner. Namely, let~$\Od$ be discrete approximations to a simply connected domain~$\Omega$ such that~$0\in\Omega$. Then, LERW obtained from simple random walks on~$\Od$ started at~$0$ and stopped when hitting~$\partial\Omega$ converge (in law) to the so-called \emph{radial} SLE(2) process in~$\Omega$. This result was generalized by Zhan~\cite{zhan-04} for multiply connected domains~$\Omega$ and also for the \emph{chordal} setup when the random walks are started at a (discrete approximation of) boundary point~$a\in\partial\Omega$ and are conditioned to exit~$\Od$ through another boundary point~$b\in\partial\Omega$. Later on, another generalization appeared in~\cite{yadin-yehudayoff}: instead of~$\delta\mathbb Z^2$ one can consider any sequence of graphs~$\Gamma^\delta$ such that the simple random walks on~$\Gamma^\delta$ converge to the Brownian motion. Since then, variants of the LERW model have become standard examples of lattice systems for which one can rigorously prove the convergence of interfaces to SLE and the Conformal Field Theory (CFT) predictions for correlation functions, e.g. see~\cite{karrila-kytola-peltola}.

In parallel with a great success of studying the (conjectural) conformally invariant limits of \emph{critical} 2D lattice models achieved during the last two decades, a program to study their near-critical perturbations was advocated by Makarov and Smirnov in 2009, with \emph{massive} LERW (mLERW) being one of the cases most amenable for the rigorous analysis, see~\cite{Makarov-Smirnov}.
On square lattice with mesh size $\delta$, given~$m>0$, {the \emph{massive random walk}  is defined as follows: at each step, the walk moves to one of the four neighboring vertices with probability $\frac{1}{4}(1-m^2\delta^2)$ or dies with probability~$m^2\delta^2$ (which is called the \emph{killing rate}).} Then, mLERW in~$\Od$ is defined by applying the same loop erasing procedure as above to massive random walks, conditioned to exit from~$\Od$ through a fixed boundary point~$b^\delta$ and not to die before this moment. The following result is given in~\cite[Theorem~2.1]{Makarov-Smirnov}:
\begin{theorem}\label{theorem}
Let~$(\Od;a^\delta,b^\delta)$ be discrete approximations to a bounded simply connected domain~$(\Omega;a,b)$ with two marked boundary points ({more accurately, degenerate} prime ends {of $\Omega$; see Remark~\ref{rem:thm11}(i) below}). For each~$m>0$ the scaling limit $\gamma$ of mLERW on~$(\Od;a^\delta,b^\delta)$ exists and is given by a chordal {stochastic} Loewner evolution process~\eqref{eq:Loewner} whose driving term $\xi_t$ satisfies the SDE
\begin{equation}
\label{eq:mSLE}
\mathrm{d}\xi_t = \sqrt{2}\mathrm{d}B_t + 2\lambda_t\mathrm{d}t,\qquad {\lambda_t = \frac{\partial}{\partial (g_t(a_t))}\log \frac{P^{(m)}_{\Omega_t}(a_t,z)}{P_{\Omega_t}(a_t,z)}\Big|_{z=b}},
\end{equation}
where~$P^{(m)}_{\Omega_t}(a_t,\cdot)$ and~$P_{\Omega_t}(a_t,\cdot)$ denote the massive and the classical Poisson kernels in the domain~$\Omega_t:=\Omega\smallsetminus\gamma[0,t]$ and the logarithmic derivative with respect to~$a_t$ is taken in the Loewner chart~$g_t:\Omega_t\to \mathbb H$; {see Remark~\ref{rem:thm11}(ii).} Moreover, \eqref{eq:mSLE} has a unique weak solution whose law is absolutely continuous with respect to~$\sqrt{2}B_t$. In other words, these scaling limits {(known under the name mSLE(2))} are absolutely continuous with respect to the classical {Schramm--Loewner Evolutions with $\kappa=2$.}
\end{theorem}

\begin{remark} \label{rem:thm11} {(i) We refer the reader to~\cite[Chapter 2]{pommerenke-book} for basic notions of the geometric function theory in what concerns the boundaries of planar domains and the correspondence between them induced by conformal maps. Loosely speaking, a \emph{degenerate prime end} of $\Omega$ should be thought of as an equivalence class of sequences of inner points converging to a point on the (topological) boundary of $\Omega$.} Although we only consider the chordal setup in this paper, the convergence of radial mLERW follows from almost the same lines {and} requires less effort since the normalization of the martingale observable near the target point {becomes} a trivial statement.

\smallskip

\noindent {(ii) We write the formula~\eqref{eq:mSLE} for the drift term~$2\lambda_t\mathrm{d}t$ in the same (slightly informal) form as it appeared in~\cite{Makarov-Smirnov}. The rigorous definition of the quantity
\begin{equation}
\label{eq:lambda=Qm/Pm(b)}
\frac{\partial}{\partial (g_t(a_t))}\log \frac{P^{(m)}_{\Omega_t}(a_t,z)}{P_{\Omega_t}(a_t,z)}\biggl|_{z=b}\ :=\ \frac{Q^{(m)}_{\Omega_t}(a_t,z)}{P^{(m)}_{\Omega_t}(a_t,z)}\biggr|_{z=b}
\end{equation}
is given in Section~\ref{sect:continuum}. The function $Q_{\Omega_t}^{(m)}(a_t,\cdot)$ (defined by~\eqref{eq:Qmdef}) can be thought of as the derivative of the massive Poisson kernel $P_{\Omega_t}^{(m)}(a_t,\cdot)$ (defined by~\eqref{eq:Pmdef}) with respect to the source point $a_t$ (after performing the uniformization $g_t:\Omega_t\to\mathbb{H}$). If $m=0$, then $Q_{\Omega_t}(a_t,z)/P_{\Omega_t}(a_t,z)\to 0$ as $z\to b$ (see~\eqref{eq:PQdef} and~\eqref{eq:Q(b)/P(b)=0}); this is why only the massive term remains in the right-hand side of~\eqref{eq:lambda=Qm/Pm(b)}.}
\end{remark}

To the best of our knowledge, no follow up of~\cite{Makarov-Smirnov} appeared since then. The goal of this paper is to provide technical details required for the proof of Theorem~\ref{theorem} as we believe that this might be of interest to the community and as we intend to pursue a rigorous understanding of further steps in the Makarov--Smirnov program ({notably,} those related to the near-critical Ising model; see~\cite[Sections~2.3 and~2.5]{Makarov-Smirnov} as well as~\cite[Question~4.12]{Makarov-Smirnov} for~$\kappa=3$). {It is worth emphasizing that the paper~\cite{Makarov-Smirnov} contains a lot of intriguing questions and conjectures which remain mostly unexplored since then, some of them most probably being very hard. One of the questions posed in~\cite{Makarov-Smirnov} is to understand which massive perturbations of the classical SLE($\kappa$) curves are absolutely continuous and which are mutually singular with respect to the unperturbed ones (e.g., see~\cite[Question~4.5]{Makarov-Smirnov}). In this regard, recall that
\begin{itemize}
\item The scaling limit of the near-critical percolation is known to be \emph{singular} with respect to the classical SLE(6) curves; see~\cite{nolin-werner}.
\item The scaling limits of the mLERW and of the massive Harmonic Explorer paths are \emph{absolutely continuous} with respect to SLE(2) and to SLE(4), respectively. (As mentioned in~\cite[Section~3.2]{Makarov-Smirnov}, the latter case can be analyzed using the same type of arguments. Though in this case the absolute continuity is less clear \emph{a~priori} from the discrete model, it can be derived \emph{a~posteriori} from the analysis of the driving process $\xi_t$; see also~\cite{shao-master}.)
\item However, the heuristics is controversial already for
    the scaling limit of the near-critical Ising model interfaces.
    For a while, this research direction was blocked by the lack of techniques allowing to prove the convergence of massive fermionic observables in rough domains (to the best of our knowledge, \cite[Sections~2.4,~2.5]{Makarov-Smirnov} had no follow-up). Such techniques were suggested in a recent work of Park~\cite{park-spin,park-FK} (see also an alternative approach to convergence theorems developed in~\cite[Section~4]{chelkak-semb}); we hope that they will allow to analyze this case in more detail.
\end{itemize}}

We now {move back to the main subject of this paper and} discuss the setup in which we prove Theorem~\ref{theorem}.
\begin{itemize}
\item $\Od$ are assumed to converge to~$\Omega$ in the \emph{Cara\-th\'eodory} topology (see Section~\ref{sect:Cara} for more details). We do not assume any regularity of~$\Omega$ (or~$\Od$) near degenerate prime ends~$a,b$, except that~$a^\delta,b^\delta$ are supposed to be \emph{close discrete approximations} of~$a,b$ in the sense of the recent paper of Karrila~\cite{karrila}. It is worth noting that in~\cite{zhan-04} it was assumed that the boundary of~$\Omega$ is `flat' near the target point~$b$, a technical restriction which was removed in~\cite{uchiyama} in the general setup of~\cite{yadin-yehudayoff}. Our approach to this technicality is based upon the tools from~\cite{chelkak-toolbox} (see Section~\ref{sect:boundarybehavior} for details), similar uniform estimates were independently obtained by Karrila~\cite[Appendix~A]{karrila-USTbranches} basing upon the conformal crossing estimates developed for the random walk in~\cite{kemppainen-smirnov}.
\item The mode of convergence of discrete random curves~$\gamma^\delta$ to continuous ones is provided by the framework of Kemppainen and Smirnov~\cite{kemppainen-smirnov} (with a recent addition of Karrila~\cite{karrila} in what concerns the vicinities of the endpoints~$a$ and~$b$), see Section~\ref{sect:topologies} for details. Namely, the weak convergence of the law of mLERW to that defined by~\eqref{eq:mSLE} holds with respect to each of the following topologies: uniform convergence of curves~$\gamma^\delta$ to~$\gamma$ after a reparametrization, convergence of conformal images~$\gamma^\delta_{\mathbb H}:=\phi_{\contOd}(\gamma^\delta)$ to~$\gamma_{\mathbb H}:=\phi_\Omega(\gamma)$ under the half-plane capacity parametrization, convergence of the driving terms~$\xi_t^\delta$ in the Loewner equations describing~$\gamma^\delta_{\mathbb H}$ to~$\xi_t$. Using the result of Lawler and Viklund~\cite{lawler-viklund} on the convergence of classical LERWs to SLE(2) in the so-called natural parametrization, one can easily deduce the same convergence for massive LERWs from our proof.
\end{itemize}

There are several known strategies to prove the convergence of discrete random curves to classical SLEs, most of them relying upon the convergence of \emph{discrete martingale observables}~$M^\delta_{(\Od;a^\delta,b^\delta)}(z)$ to~$M_{(\Omega;a,b)}(z)$ as~$(\Od;a^\delta,b^\delta)\to(\Omega;a,b)$; see~\eqref{eq:defMm} for the definition of these observables in the LERW case. The approach used in the original papers~\cite{LSW-LERW,zhan-04} on the subject (see also~\cite{izyurov-multiple} for similar considerations in the Ising model context) relies upon the Skorohod embedding theorem and an approximate version of the L\'evy characterization of the Brownian motion. A different viewpoint was advocated by Smirnov in~\cite{smirnov-icm06}: once the tightness framework of~\cite{kemppainen-smirnov} is set up, one gets the martingale property of~$\xi_t$ and its quadratic variation from coefficients of the asymptotic expansion of~$M_{(\Omega_t;a_t,b)}(z)$ near the \emph{target} point~$b$, e.g. see~\cite[Section~4.4]{smirnov-icm06} or~\cite[Section~6.3]{duminil-smirnov-clay} for sample computations. (Note however that~\cite{zhan-04} and~\cite{izyurov-multiple} rely upon asymptotics of~$M_{(\Omega_t;a_t,b)}(z)$ near the \emph{source} point~$a_t$, which are known to be more useful in the multiple SLE context.)

In the massive setup, one does not have conformal invariance, which makes these asymptotics of~$M_{\Omega_t}(a_t,z)$ rather sensitive to the local geometry of~$\Omega_t$ near~$b$ or~$a_t$. Moreover, even if we assume that the boundary of~$\Omega$ is flat near~$b$, these asymptotics are written in terms of Bessel functions instead of powers of~$(z\!-\!b)$.
In this paper we use a combination of the two strategies: we do rely upon the tightness framework of~\cite{kemppainen-smirnov} but analyze the stochastic processes~$M_{(\Omega_t;a_t,b)}(z)$ at \emph{fixed} points~$z\in\Omega_t$ instead of discussing their asymptotics; cf.~\cite{hongler-kytola} or~\cite[Section~3.1]{izyurov-multiple}.

In conformally invariant setups, it is known (e.g., see~\cite{werner-percolation-book} or~\cite{hongler-kytola}) that one can easily derive the fact that the process~$\xi_t$ is a continuous semi-martingale directly from the fact that~$M_{(\Omega_t;a_t,b)}(z)$ are continuous (local) martingales, using explicit representations of those via~$\xi_t$. We illustrate this idea in Section~\ref{sect:conv_m=0} when discussing the convergence of the classical LERW to SLE(2). Despite the lack of explicit formulas, similar arguments \emph{can} be used in the massive setup though being more involved. Nevertheless, we prefer to follow a more conceptual approach suggested in~\cite{BBK,BBC} and~\cite{Makarov-Smirnov}, which relies upon the Girsanov theorem and the fact that mLERW can (and, arguably, should) be viewed as the classical LERW weighted by an appropriate density caused by the killing rate; in this approach the fact that~$\xi_t$ is a semi-martingale does not require any special proof (see Section~\ref{sect:abs-cont}).

Certainly, the idea of weighting SLE curves by martingales dates back to the very first developments in the subject, e.g. see~\cite{dubedat-commSLE,schramm-wilson} or~\cite{wu-hypergeomSLE,kemppainen-smirnov-iv} for more recent examples. Nevertheless, there exist an important difference between the `critical/critical' and `the `massive/critical' contexts. In the setup of Theorem~\ref{theorem}, the density of mSLE(2) with respect to the classical SLE(2) does \emph{not} coincide with the ratio of regularized partition functions~$P^{(m)}_{\Omega_t}(a_t,b)/P_{\Omega_t}(a_t,b)$ in~$\Omega_t:=\Omega\smallsetminus\gamma[0,t]$. The reason is that the total mass of massive RW loops attached to the tip~$a_t$ is strictly smaller than the mass of the critical ones, which results in a (positive) drift of this ratio; see
also~\cite[Section~4]{BBK} for a discussion of this effect from the theoretical physics perspective. Nevertheless, the expression for the drift {term~$2\lambda_t\mathrm{d}t$} in~\eqref{eq:mSLE} has exactly the same structure as in `critical/critical' setups, see Remark~\ref{rem:Nm-vs-Nb} for additional comments.

The rest of the article is organized as follows. In Section~\ref{sect:preliminaries} we collect preliminaries and discuss the absolute continuity of mLERW with respect to LERW and that of their scaling limits. In Section~\ref{sect:convergence} we prove the convergence of discrete martingale observables as~$\delta\to 0$. Section~\ref{sect:continuum} is devoted to a priori estimates and computations in continuum. The proof of Theorem~\ref{theorem} is given at the end of the paper.

\subsection*{Acknowledgements} This research was supported by the ANR-18-CE40-0033 project DIMERS. Dmitry Chelkak is grateful to Stanislav Smirnov for explaining the ideas of~\cite{Makarov-Smirnov} during several conversations dating back to 2009--2014. We want to thank Michel Bauer, Konstantin Izyurov and Kalle Kyt\"ol\"a for valuable comments and for encouraging us to write this paper; Alex Karrila for useful discussions of his research~\cite{karrila,karrila-USTbranches}; Chengyang Shao for discussions during his spring 2017 internship at the ENS~\cite{shao-master}; and Mikhail Skopenkov for a feedback, which in particular included pointing out a mess at the end of the proof of~\cite[Theorem~3.13]{chelkak-smirnov-11}. This proof is sketched in Section~\ref{sect:convergence} with a necessary correction. {Last but not least we want to thank the referees for their careful reading of the paper and for useful comments and suggestions on the presentation of the material.}

\section{Preliminaries} \label{sect:preliminaries}

\subsection{Discrete domains, partition functions and martingale observables}\label{sect:notation}
Let $\Omega\subset\mathbb{C}$ be a bounded simply connected domain with two marked degenerate prime ends $a,b$. We approximate $(\Omega;a,b)$ by simply connected \emph{subgraphs}~$\Od$ of the square grids~$\delta\mathbb{Z}^2$ and their boundary vertices~$a^\delta,b^\delta$. More precisely, to each simply connected graph~$\Od\subset\delta\mathbb{Z}^2$ we associate an open simply connected \emph{polygonal domain}~$\widehat{\Omega}^\delta\subset\mathbb{C}$ by taking the union of all open~$2\delta\times 2\delta$ squares centered at vertices of~$\Od$. Note that the boundary of~$\widehat{\Omega}^\delta$ consists of edges of~$\delta\mathbb{Z}^2$; see Fig.~\ref{fig:domain} for an illustration. We set $\Int\Omega^\delta:=V(\Omega^{\delta})$ and define {the boundary $\partial\Od$ of $\Od$ as}
 \begin{equation}
 \label{x:paO}
 \partial \Omega^{\delta} := \{(v;(v_{\mathrm{int}},v)):\ v \notin \Int\Omega^{\delta},\ v\sim v_{\mathrm{int}},\ v_{\mathrm{int}} \in \Int\Omega^{\delta}\},
 \end{equation}
{here and below the notation $v\sim v'$ means that the vertices $v,v'\in\mathbb{Z}^2$ are adjacent to each other.} (The reason for this definition of~$\partial \Omega^\delta$ is that the same vertex~$v$ may be connected to several points $v_{\mathrm{int}}\in\Int \Omega^{\delta}$. When talking about exiting events of random walks, all such edges~$(v_{\mathrm{int}},v)$ correspond to different possibilities to exit~$\Omega^\delta$.) Usually, we slightly abuse the notation and treat $\partial \Omega^{\delta}$ as a set of~$v\in \delta\mathbb Z^2$ without indicating the outgoing edges~$(v_{\mathrm{int}},v)$ if no confusion arises. Sometimes we also use the notation~$\overline{\Omega}{}^\delta:=\Omega^\delta\cup\partial\Omega^\delta$.

\smallskip

Given~$0<\delta<m^{-1}\le +\infty$, a discrete domain~$\Omega^\delta\subset\delta\mathbb Z^2$, and two \emph{interior or boundary} vertices~$w^\delta,z^\delta$, we define the \emph{partition function} of massive random walks running from~$w^\delta$ to~$z^\delta$ in~$\Omega^\delta$ as
\begin{equation}
\label{eq:defZm}
\Zm_{\Omega^{\delta}}(w^\delta,z^\delta)\ :=\ \!\!\sum_{\pi^\delta \in S_{\Omega^{\delta}}(w^\delta;z^\delta)}\!\!\bigl(\tfrac{1}{4}(1-m^2\delta^2)\bigr)^{\#\pi^{\delta}},\quad w^\delta,z^\delta\in\overline{\Omega}{}^\delta,
\end{equation}
where $S_{\Omega^{\delta}}(w^\delta;z^\delta)$ denotes the set of all lattice paths connecting $w^\delta$ and $z^\delta$ inside~$\Omega^{\delta}$, and $\#\pi^{\delta}$ is the number of \emph{interior} edges of~$\Od$ in $\pi^{\delta}$. (In other words, we do \emph{not} count the edges~$(w^\delta,w^\delta_{\mathrm{int}})$ and~$(z^\delta_{\mathrm{int}},z^\delta)$ in~$\#\pi^\delta$ if~$w^\delta\in\partial\Od$ and/or~$z^\delta\in\partial\Od$.)  To simplify the notation, we drop the superscript $(m)$ when speaking about random walks without killing (i.e., $m=0$). {Below we often rely upon the following identity, which relates the partition functions $\Z^{(m)}_{\Omega^\delta}$ and $\Z_{\Omega^\delta}$.

\begin{lemma} \label{lem:Zm(w,z)=Z-ZmZ}
Given a discrete domain $\Omega^\delta$, two points \mbox{$z^\delta,w^\delta\in\overline{\Omega}{}^{\delta}$}, and \mbox{$m\in (0,\delta^{-1})$}, we have
	\begin{align}
\notag
 (1-m^2\delta^2)\cdot\Zm_{\Od}(w^\delta,z^\delta)\ =\ &\Z_{\Od}(w^\delta,z^\delta)\\
 &- m^2\delta^2\sum\nolimits_{v^\delta\in \Int\Od}\Z_{\Od}(w^\delta,v^\delta)\Zm_{\Od}(v^\delta,z^\delta).
 \label{eq:Zm(w,z)=Z-ZmZ} 
\end{align}
\end{lemma}

\begin{proof} Recall that both $\Zm_{\Od}(w^\delta,z^\delta)$ and $\Z_{\Od}(w^\delta,z^\delta)$ are defined as sums over random walk trajectories $\pi^\delta\in S_{\Od}(w^\delta;z^\delta)$ running from $w^\delta$ to $z^\delta$ inside $\Omega^\delta$. Also,
by splitting $\pi^\delta$ into two parts (from~$w^\delta$ to~$v^\delta$ and from~$v^\delta$ to~$z^\delta$) and summing over all $\#\pi^\delta\!+\!1$ possible choices of~$v^\delta$, one easily sees that
\begin{align*}
\sum_{v^\delta\in\Int\Omega^\delta}\Z_{\Omega^\delta}(w^\delta,v^\delta)\Z^{(m)}_{\Omega^\delta}(v^\delta,z^\delta)\ &=\!\! \sum_{\pi^\delta\in S_{\Omega^\delta}(w^\delta,z^\delta)}\sum_{k=0}^{\#\pi^\delta}\ \big(\tfrac{1}{4}\big)^{k}\big(\tfrac{1}{4}(1-m^2\delta^2)\big)^{{\#\pi^\delta}-k}\\
&=\!\!\sum_{\pi^\delta\in S_{\Omega^\delta}(w^\delta,z^\delta)}\big(\tfrac{1}{4}\big)^{\#\pi^\delta}\cdot\frac{1-(1-m^2\delta^2)^{\#\pi^\delta+1}}{m^2\delta^2}\,.
\end{align*}
Thus, the identity~\eqref{eq:Zm(w,z)=Z-ZmZ} directly follows from the definition~\eqref{eq:defZm}.
\end{proof}}

\begin{figure}
	\includegraphics[clip, width=0.6\textwidth]{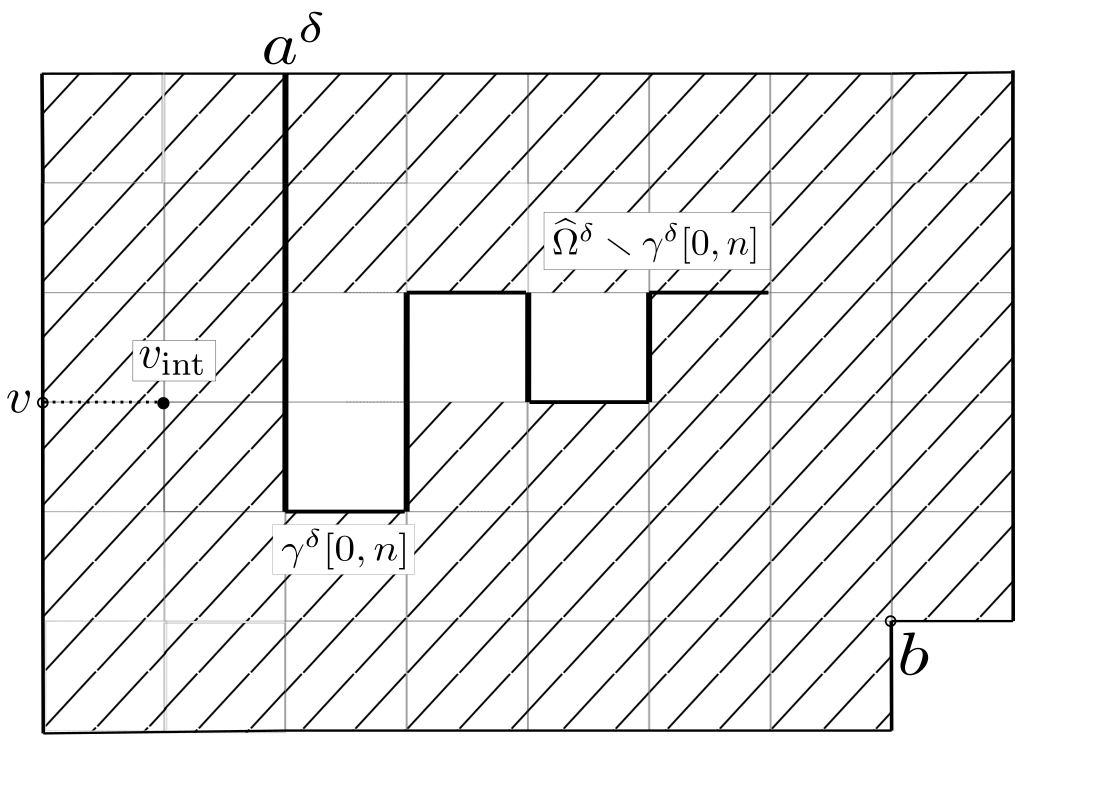}
	\caption{Discrete domain~$\Omega^\delta$, an example of a boundary vertex~$(v;(v_\mathrm{int},v))$, and a slit~$\gamma^\delta[0,n]$. The shaded area is the polygonal representation of the subgraph~$\Omega^\delta\smallsetminus\gamma^\delta[0,n]\subset\delta\mathbb{Z}^2$. Though this {polygonal domain} does not coincide with~$\contOd\smallsetminus\gamma^\delta[0,n]$, these two domains are close to each other in the Carath\'eodory sense (with respect to inner points of~$\contOd$ lying near~$b$).\label{fig:domain}}
\end{figure}

Let~$\gamma^\delta$ be a sample of the (massive or massless) LERW path from~$a^\delta$ to~$b^\delta$ in~$\Od$. We denote by~$\Omega^\delta\smallsetminus\gamma^\delta[0,n]$ the \emph{connected component} of this graph containing~$b^\delta$; see Fig.~\ref{fig:domain}. Let a sequence of vertices~$o^\delta{\in\Int\Od}$ be fixed so that~$o^\delta\to 0$ as~$\delta\to 0$. A classical argument (e.g., see~\cite[Remark~3.6]{LSW-LERW}) implies that, for each~$v^\delta\in\Int\Od$, the function
\begin{equation}
\label{eq:defMm}
M^{(m),\delta}_{\Od\smallsetminus\gamma^\delta[0,n]}(v^\delta)\ :=\ \frac{\Zm_{\Od\smallsetminus\gamma^\delta[0,n]}(\gamma^\delta(n),v^\delta)}{\Zm_{\Od\smallsetminus\gamma^\delta[0,n]}(\gamma^\delta(n),b^\delta)}\cdot \Z_{\Od}(o^\delta,b^\delta),
\end{equation}
is a martingale with respect to the filtration $\mathcal{F}_n:=\sigma(\gamma^\delta[0,n])$ generated by first~$n$ steps of $\gamma^\delta$, until~$v^\delta$ is hit by~$\gamma^\delta$ or disconnected from~$b^\delta$. The additional normalization factor~$\Z_{\Od}(o^\delta,b^\delta)$ does not depend neither on~$\gamma^\delta$ nor on~$m$ and is introduced for further convenience. {Note that the behavior of this factor (which is nothing but the harmonic measure of $b^\delta$ in $\Od$ viewed from $o^\delta$) as $\delta\to 0$ can be very irregular as we do not require much about the behavior of the boundary $\partial\Od$ near $b^\delta$; the role of this normalization is to compensate the similar irregularity in the behavior of the denominator of~\eqref{eq:defMm}.} As in the notation for partition functions, we drop the superscript~$(m)$ in~\eqref{eq:defMm} when speaking about classical ($m=0$) LERW.

\subsection{Carath\'eodory convergence of~\texorpdfstring{$\bm{\Omega^\delta}$}{} and reparametrization by capacity}\label{sect:Cara} Throughout this paper we assume that all domains under consideration are uniformly bounded (that is, are contained in some $B(0,R)$ for a fixed $R>0$) and that $0$ is contained in all domains. Let~$\phi_\Omega:\Omega\to\mathbb H$ be a conformal uniformization of~$\Omega$ onto the upper half-plane~$\mathbb H$ such that
\begin{equation}
\label{eq:defPhi}
\phi_\Omega(a)=0,\quad\phi_\Omega(b)=\infty,\quad \text{and}\quad \Im\phi_\Omega(0)=1,
\end{equation}
note that these conditions define~$\phi_\Omega$ uniquely and that one has
\begin{equation}
\label{eq:PsimG}
G_\Omega(0,z)\ =\ \frac{1}{2\pi}\log\biggl|\frac{\phi_\Omega(z)-\phi_\Omega(0)}{\phi_\Omega(z)-\overline{\phi_\Omega(0)}}\biggr|\ \sim\ -\frac{1}{\pi}\Im\frac{1}{\phi_\Omega(z)}\quad \text{as}\quad z\to b.
\end{equation}
We assume that discrete approximations $(\widehat\Omega^{\delta};a^{\delta},b^\delta)$, with~$b^\delta=b$, converge to $(\Omega;a,b)$ \emph{in the Carath\'eodory sense}, which means that (e.g., see~\cite[Chapter~1]{pommerenke-book})
\begin{itemize}
\item each inner point~$z\in\Omega$ belongs to~$\contOd$ for small enough $\delta$;
\item each boundary point~$\zeta\in\partial\Omega$ can be approximated by~$\zeta^\delta\in\partial\contOd$ as~$\delta\to0$.
\end{itemize}
Further, we require that~$a$ and~$b$ are degenerate prime ends of~$\Omega$ and that~$a^\delta$ (resp.,~$b^\delta$) is a \emph{close approximation} of~$a$ (resp., of~$b$) as defined by Karrila~\cite{karrila}:
\begin{itemize}
\item $a^\delta\to a$ as~$\delta\to 0$ and, moreover, the following is fulfilled:
\item Given $r>0$ small enough, let~$S_r$ be the arc of~$\partial B(a,r)\cap\Omega$ disconnecting (in~$\Omega$) the prime end~$a$ from~$0$ and from all other arcs of this set; in other words, $S_r$ is the last arc from a (possibly countable) collection~$\partial B(a,r)\cap\Omega$ to cross for a path running from~$0$ to~$a$ inside~$\Omega$. We require that, for each~$r$ small enough and for all sufficiently (depending on~$r$) small~$\delta$, the boundary point~$a^\delta$ of~$\Od$ is connected to the midpoint of~$S_r$ inside~$\contOd\cap B(a,r)$.
\end{itemize}
We fix a uniformization~$\phi_{\contOd}:\contOd\to{\mathbb H}$ similarly to~\eqref{eq:defPhi} so that
\[
\phi_{\contOd}(a^\delta)=0,\quad \phi_{\contOd}(b^\delta)=\infty,\quad \text{and}\quad \Im\phi_{\contOd}(0)=1,
\]
note that the Carath\'eodory convergence of~$\contOd$ to~$\Omega$ can be reformulated as
\begin{equation}\label{x:Cara}
\begin{array}{ll}
\phi_{\contOd}\ \to\ \phi_\Omega&  \text{uniformly on compact subsets of $\Omega$},\\[2pt]
\phi_{\contOd}^{-1}\ \to\ \phi_\Omega^{-1}&\text{uniformly on compact subsets of $\mathbb H$}.
\end{array}
\end{equation}

From now onwards we assume (without loss of generality) that the discrete approximations~$\contOd$ are shifted slightly so that the target point~$b^\delta=b$ is always the same. Inside all {polygonal domains $\widehat\Omega^{\delta}$} (and similarly inside~$\Omega$), one can define the inner distance to the prime end~$b$ and the $r$-vicinities of~$b$ as follows:
 \begin{align}
 \notag \rho_{\contOd}(b,z) & :=\inf\{r>0: z~\text{and}~b~\text{are connected in}~{\contOd}\cap B_{\mathbb C}(b,r)\},\\
 \label{eq:rhoB-def} B_{\contOd}(b,r) & :=\{z\in\contOd: \rho_{\contOd}(b,z)<r\}.
 \end{align}
 Note that~$\rho_\Omega(b,z)$ is a continuous function of~$z\in\Omega$. Moreover,
 \begin{equation}
 \label{x:rho<r}
 \rho_\Omega(b,z)<r\quad\Rightarrow\quad \rho_{\contOd}(b,z)<r\quad \text{for small enough}~\delta
 \end{equation}
 since a path connecting~$z$ to~$b$ inside~$\Omega\cap B_{\mathbb C}(b,r)$ eventually belongs to~${\contOd}$ except, possibly, a tiny portion near~$b$. As we assume that~$b^\delta$ is a close approximation of the prime end~$b$, the implication~\eqref{x:rho<r} follows.

 \smallskip

 Let~$\gamma^\delta_{\mathbb H}:=\phi_{\contOd}(\gamma^\delta)$ be the conformal images of LERW trajectories~$\gamma^\delta$, considered as continuous paths in the upper half-plane~$\mathbb H$. These continuous simple curves can be canonically parameterized by the so-called \emph{half-plane capacity} of their initial segments. Namely, a uniformization map~$g_t:\mathbb H\smallsetminus\gamma^\delta_{\mathbb H}[0,t]\to\mathbb H$ normalized at infinity is required to have the asymptotics~$g_t(z)=z+{2t}z^{-1}+O(|z|^{-2})$ as~$|z|\to\infty$.

 Given~$t>0$ we define a random variable~$n^\delta_t$ to be the first integer such that the half-plane capacity of~$\phi_{\contOd}(\gamma^\delta[0,n])$ is greater or equal than~$t$. Further, given a small enough~$r>0$ we define~$n^\delta_{t,r}$ to be the minimum of~$n^\delta_t$ and the first integer such that~$\gamma^\delta(n)\in B_{\contOd}(b,r)$. Clearly, both~$n^\delta_t$ and~$n^\delta_{t,r}$ are stopping times with respect to the filtration~$\mathcal{F}_n:=\sigma(\gamma^\delta[0,n])$. We set $\Omega^\delta_t$ (resp. $\Omega^\delta_{t,r}$) to be the connected component of $\Omega^\delta\smallsetminus\gamma^\delta[0,n^\delta_t]$ (resp. $\Omega^\delta\smallsetminus\gamma^\delta[0,n^\delta_{t,r}]$) including $b$ and $a^\delta_t:=\gamma^\delta(n^\delta_t)$ (resp. $a^\delta_{t,r}:=\gamma^\delta(n^\delta_{t,r}))$.

 The following lemma guarantees that the change of the parametrization from integers~$n^\delta_t$ to the half-plane capacity~$t$ does not create big jumps. The proof given below is based upon compactness arguments though one can use standard estimates (e.g., see~\cite[Proposition~6.5]{berestycki-norris}) of capacity increments in the upper half-plane~$\mathbb H$ instead. However, it is worth noting that one does not have an immediate a priori bound of~$\mathrm{diam}(\gamma^\delta_\mathbb H[0,n^\delta_{t,r}])$ in the situation when the curve~$\gamma^\delta$ approaches~$b$ along the boundary of~$\Omega^\delta$, which might require to introduce additional stopping times to handle this scenario explicitly.

 \begin{lemma} \label{lem:capacity-increments}
 Let~$(\contOd;a^\delta,b^\delta)$ approximate~$(\Omega;a,b)$ as described above. Then, for each~\mbox{$r>0$}, the increments of the half-plane capacities of the slits~$\phi_{\contOd}(\gamma^\delta[0,n])$ are uniformly (in {both}~$\gamma^\delta$ {and $n$}) small as~$\delta\to 0$ provided that~$\gamma^\delta[0,n]$ do not enter the vicinities~$B_{\contOd}(b,r)$ of the target point~$b$. In particular, the capacities of the slits $\phi_{\contOd}(\gamma^\delta[0,n^\delta_{t,r}])$ are uniformly bounded by~$t\!+\!o(1)$ as~$\delta\to 0$.
 \end{lemma}

 \begin{proof}
 The set of all simply connected domains~$\contOd\smallsetminus\gamma^\delta[0,n]$ under consideration is precompact in the Carath\'eodory topology (with respect to points near~$b$). Suppose on the contrary that the one-step increments of the half-plane capacities of~$\phi_{\contOd}(\gamma^\delta[0,n])$ do \emph{not} vanish as~$\delta\to0$ for a sequence of curves~$\gamma^\delta[0,n^\delta]$ such that~$\gamma^\delta[0,n^\delta\!-\!1]\cap B_{\contOd}(b,r)=\emptyset$. By compactness, one can find a subsequence along which~${\contOd}\smallsetminus\gamma^\delta[0,n^\delta]$ converge in the Carath\'eodory sense (with respect to points near~$b$). Clearly, ${\contOd}\smallsetminus\gamma^\delta[0,n^\delta\!-\!1]$ converge to the same limit and hence one can find conformal homeomorphisms
 \[
 {\contOd}\smallsetminus \gamma^\delta[0,n^\delta]\ \to\ {\contOd}\smallsetminus \gamma^\delta[0,n^\delta\!-\!1]
 \]
 that become arbitrary close to the identity on each {compact subset} $K\subset B_\Omega(b,r)$, note that one necessarily has~$K\subset B_{\contOd}(b,r)$ for small enough~$\delta$ due to~\eqref{x:rho<r}. Due to~\eqref{x:Cara}, this implies that the conformal maps
 \[
 \mathbb H\smallsetminus \phi_{\contOd}(\gamma^\delta[0,n^\delta])\ \mathop{\longrightarrow}\limits^{\phi_{\contOd}^{-1}}\ {\contOd}\smallsetminus \gamma^\delta[0,n^\delta]\ \to\ {\contOd}\smallsetminus \gamma^\delta[0,n^\delta\!-\!1] \ \mathop{\longrightarrow}\limits^{\phi_{\contOd}}\ \mathbb H\smallsetminus \phi_{\contOd}(\gamma^\delta[0,n^\delta\!-\!1])
 \]
 become (as~$\delta\to0$) arbitrary close to the identity on compact subsets of the fixed vicinity $\phi_\Omega(B_\Omega(b,r))$ of~$\infty$ in the upper half-plane. This contradicts to the assumption that the half-plane capacities of~$\phi_{\contOd}(\gamma^\delta[0,n^\delta\!-\!1])$ and~$\phi_{\contOd}(\gamma^\delta[0,n^\delta])$ differ by a constant amount as~$\delta\to 0$.
 \end{proof}

\subsection{Chordal SLE(2) and topologies of convergence}\label{sect:topologies}

We now discuss a few basic facts on the construction of SLE curves, the interested reader is referred to~\cite{berestycki-norris,kemppainen-book,lawler-book} for more details. Let $\gamma_{\mathbb H}$ be a continuous non-self-crossing curve in the upper half-plane $\mathbb{H}:=\{z \in \mathbb{C}:\text{Im}z >0\}$, growing from~$0$ to~$\infty$. Let~$\mathbb H\smallsetminus K_t$ denote the connected component of $\mathbb{H}\smallsetminus\gamma_{\mathbb H}[0,t]$ containing~$\infty$ (if~$\gamma_{\mathbb H}$ is not only non-self-crossing but also \emph{non-self-touching}, then~$K_t=\gamma_{\mathbb H}[0,t]$). Assume that~$\gamma_{\mathbb H}$ is parameterized by half-plane capacity so that the conformal map $g_t:\mathbb H\smallsetminus K_t\to\mathbb H$ (normalized at~$\infty$) has the asymptotics $g_t(z) =z+2tz^{-1}+O(|z|^{-2})$ as~$|z|\to\infty$. Then there exists a unique real-valued function $\xi_t$, called the \emph{driving term}, such that the following equation, called the \emph{Loewner evolution equation}, is satisfied:
\begin{equation}\label{eq:Loewner}\partial_tg_t(z) = \frac{2}{g_t(z) - \xi_t}\quad \text{for~all}~z\in\mathbb H\smallsetminus K_t,
\end{equation}
where we use the shorthand notation~$\partial_t$ for the partial derivative in~$t$. Vice versa, given a nice function~$\xi_t$ one can reconstruct the growing family~$K_t$ and, further (under some assumptions on~$\xi_t$), the curve~$\gamma_{\mathbb H}$ by solving~\eqref{eq:Loewner} with~$g_0(z)=z$.

Classical $\mathrm{SLE}_{\mathbb H}(2)$ curves in the upper half-plane correspond to \emph{random} driving terms $\xi_t=\sqrt{2}B_t$, where $(B_t)_{t \geq 0}$ is a standard Brownian motion. It is known that
\begin{itemize}
\item almost surely, $\mathrm{SLE}_{\mathbb H}(2)$ is a simple curve in the upper half plane~$\mathbb H$, see~\cite{rohde-schramm};
\item almost surely, the Hausdorff dimension of $\mathrm{SLE}_{\mathbb H}(2)$ is equal to~$\frac{5}{4}$, see~\cite{beffara-dimSLE}. Moreover, one can use the corresponding Minkowski content of the initial segments of~$\mathrm{SLE}_{\mathbb H}(2)$ to introduce the so-called \emph{natural parametrization} of these curves, see~\cite{lawler-rezaei}.
\end{itemize}

Generally, given a simply connected domain~$\Omega$ with boundary points (prime ends) $a,b\in\Omega$, chordal $\mathrm{SLE}_\Omega$ curves from~$a$ to~$b$ in~$\Omega$ are defined as preimages of~$\mathrm{SLE}_{\mathbb H}$ under a conformal uniformization~$\phi_\Omega:\Omega\to\mathbb H$ satisfying~$\phi_\Omega(a)=0$ and~$\phi_\Omega(b)=\infty$. Note that this definition does no require to fix a normalization of~$\phi_\Omega$ due to the scale invariance of the law of $\mathrm{SLE}_{\mathbb H}$ curves.

\smallskip

When speaking about the tightness of random curves in~$(\Omega^\delta;a^\delta,b^\delta)$ we rely upon a powerful framework developed by Kemppainen and Smirnov in~\cite{kemppainen-smirnov} as well as upon a recent work of Karrila~\cite{karrila} (in which the behaviour in vicinities of the endpoints~$a,b$ is discussed). Let~$\xi^\delta$ be a random driving term corresponding via~\eqref{eq:Loewner} to the conformal images~$\gamma^\delta_{\mathbb H}:=\phi_{\contOd}(\gamma^\delta)$ of LERWs in~$(\Od;a^\delta,b)$. It is known since the work~\cite{aizenman-burchard} of Aizenman and Burchard (see also~\cite{LSW-LERW}) that appropriate~\emph{crossing estimates} imply that
\begin{enumerate}
\item the family of random curves~$\gamma^\delta$ (except maybe in vicinities of endpoints) is tight in the topology induced by the metric
    $\min_{\psi_1,\psi_2}\|\gamma_1\circ\psi_1-\gamma_2\circ\psi_2\|_\infty$, with minimum taken over all parametrizations~$\psi_1,\psi_2$ of two curves~$\gamma_1,\gamma_2$.
\end{enumerate}
The results of Kempainen and Smirnov (see~\cite[Theorem~1.5 and~Corollary 1.7]{kemppainen-smirnov} as well as~\cite[Section~4.5]{kemppainen-smirnov} where the required crossing estimates are checked for the loop-erased random walks) give much more:
\begin{enumerate}\setcounter{enumi}{1}
	\item the driving terms $\xi^{\delta}$ are tight in the space of continuous functions on $[0,\infty)$ with topology of uniform convergence on compact intervals $[0,T]$;
    \item the curves~$\gamma^\delta_{\mathbb H}$ are tight in the same topology as in (1);
    	\item the curves $\gamma^\delta_{\mathbb H}$, parameterized by capacity, are tight in the space of continuous functions on $[0,\infty)$ with topology of uniform convergence on~$[0,T]$.
\end{enumerate}
Moreover, a weak convergence in one of the topologies (2)--(4) imply the convergence in two others. Furthermore, provided that $(\widehat\Omega^{\delta};a^{\delta},b)$ converge to~$(\Omega;a,b)$ in the Carath\'eodory sense so that $a^{\delta}$ and $b^\delta=b$ are close approximations of degenerate prime ends~$a$ and~$b$ of~$\Omega$, the following holds:
\begin{enumerate}\setcounter{enumi}{4}
\item if a sequence of random curves $\gamma^{\delta}_{\mathbb H}$ converges weakly in the topologies \mbox{(2)--(4)} to a random curve~$\gamma_{\mathbb H}$ then $\gamma^{\delta}$ also converges weakly to a random curve which, almost surely, is supported on the limiting domain~$\Omega$ due to \cite[Corollary~1.8]{kemppainen-smirnov}, and has the same law as~$\phi_\Omega^{-1}(\gamma_{\mathbb H})$ due to \cite[Theorem~4.4]{karrila}.
\end{enumerate}

\subsection{Convergence of classical LERW to chordal SLE(2)} \label{sect:conv_m=0} To keep the presentation self-contained, in this section we sketch (a variant of the strategy used in~\cite{hongler-kytola,werner-percolation-book}) a proof of the classical result: convergence of the usual loop-erased random walks to SLE(2), in the setup of Theorem~\ref{theorem} discussed in the introduction.

As discussed above, the family of LERW probability measures on $(\Omega^{\delta};a^{\delta},b)$ is tight, provided that the curves~$\gamma^\delta$ are parameterized by the half-plane capacities of their conformal images~$\phi_{\contOd}(\gamma^\delta)$ in $(\mathbb{H};0,\infty)$. Since the space of continuous functions is metrizable and separable, by Skorokhod representation theorem, we can suppose that for each weakly convergent subsequence of these measures we also have~$\gamma^\delta\to\gamma$ almost surely.

Let~$\tau_r:=\inf\{t>0: \gamma(t)\in B_\Omega(b,r)\}$ and~$\tau_r^\delta$ be the similar stopping times (in the half-plane capacity parametrization) for the discrete curves~$\gamma^\delta$. {It is not hard to see that for each (as for now, unknown) law $\mathbb{P}$ of $\gamma$ on the set of continuous parameterized curves, the following statement holds:
\begin{equation}
\label{eq:tdr->tr}
\text{for almost all~$r>0$ one almost surely has $\tau^\delta_r\to\tau_r$.}
\end{equation}
To prove~\eqref{eq:tdr->tr}, let us consider a continuous process~$t\mapsto \rho_t:=\rho_\Omega(b,{\gamma(t)})$. Since the curves~$\gamma^\delta$ converge to~$\gamma$ in the capacity parametrization, one has~$\tau^\delta_{r}\to\tau_{r}$ unless~$\rho_t$ has a local minimum at level~$r$. (Indeed, note that the inequality \mbox{$\limsup_{\delta\to 0}\tau^\delta_r\le \tau_r$} is trivial: if $\gamma$ enters the open set $B_\Omega(b,r)$, then so $\gamma^\delta$ (with small enough $\delta$) do; no later than approximately at the same time. On the other hand, if $\rho_t$ does \emph{not} have a local minimum at level $r$, then for each $\varepsilon>0$ one can find $\eta(\varepsilon)>0$ such that $\rho_t\ge r+\eta(\varepsilon)$ for all $t\le \tau_r-\varepsilon$, which gives $\liminf_{\delta\to 0}\tau^\delta_r\ge \tau_r-\varepsilon$.) The set of locally minimal values of a continuous function $\rho_t$ is at most countable since each such a value $r$ is the minimum of $\rho_t$ over a rational interval. In particular,
\[
\mu_{\mathrm{Leb}}(\{r>0:\rho_t~\text{has~a~local~minimum~at~level~$r$}\})=0,
\]
for each continuous function $t\mapsto \rho_t$ and thus (almost) surely in the context of random processes under consideration. Therefore,
\[
\mathbb{P} [\,\text{the~process~$\rho_t$~has~a~local~minimum~at~level~$r$}\,]=0\ \ \text{for almost all~$r>0$,}
\]
due to the Fubini theorem for the product measure $\mathbb{P}\times \mu_\mathrm{Leb}$, which implies~\eqref{eq:tdr->tr}.}

Let~$t>0$ and assume that~$r>0$ is chosen {according to~\eqref{eq:tdr->tr}} so that, almost surely,~$\tau^\delta_{s,r}:= s\wedge \tau^{\delta}_r\to s\wedge \tau_r$ and hence~$\gamma^\delta[0,n^\delta_{s,r}]\to\gamma[0,s\wedge\tau_r]$ for all~$s\in[0,t]$; see Lemma~\ref{lem:capacity-increments}. Let
\[
v\in B_\Omega(b,\tfrac{1}{2}r).
\]
The martingale property of the discrete observables~\eqref{eq:defMm} 
gives
\begin{equation}
\label{eq:discrete-mart-ts}
\mathbb E\big[M^\delta_{\Od_{t,r}}(v^\delta)f(\gamma^\delta[0,n^\delta_{s,r}])\big]
\  =\ \mathbb E\big[M^\delta_{\Od_{s,r}}(v^\delta)f(\gamma^\delta[0,n^\delta_{s,r}])\big],
\end{equation}
where $f$ is a bounded continuous test function on the space of curves. We now pass to the limit (as~$\delta\to 0$) in this identity using the following two facts:
\begin{itemize}
\item If~$\gamma^\delta[0,n^\delta_{t,r}]\to\gamma[0,t\wedge\tau_r]$, then
\begin{equation}\label{x:Pdef}
M^\delta_{\Od_{t,r}}(v^\delta)\ \to\ P_{\Omega\smallsetminus\gamma[0,t\wedge\tau_r]}(v):=-\frac{1}{\pi}\Im\frac{1}{g_{t\wedge\tau_r}(\phi(v))-\xi_{t\wedge\tau_r}}
\end{equation}
as~$\delta\to 0$. We discuss such convergence results in Section~\ref{sect:convergence} (see Proposition~\ref{prop:M=0conv} for this concrete statement).

\smallskip

\item The martingale observables are uniformly (with respect to~$\delta$ and all possible realisations of~$\gamma^\delta[0,n^\delta_{t,r}]$) bounded. Indeed, Lemma~\ref{lem:ZZ/Z(a,b)<c} implies that
\[
M^\delta_{\Od_{s,r}}(v^\delta)\ =\ \frac{\Z_{\Od_{s,r}}(a^\delta_{s,r},v^\delta)}{\Z_{\Od_{s,r}}(a^\delta_{s,r},b)}\cdot \Z_{\Od}(o^\delta,b)\ \le\ \const\cdot \frac{\Z_{\Od}(o^\delta,b)}{\Z_{\Od_{s,r}}(v^\delta,b)}
\]
with a universal multiplicative constant and
\[
\frac{\Z_{\Od}(o^\delta,b)}{\Z_{\Od_{s,r}}(v^\delta,b)}\ \le\ \frac{\Z_{\Od}(o^\delta,b)}{\Z_{B_{\Od}(b,r)}(v^\delta,b)}\ \to\ \frac{G_\Omega(0,b)}{G_{B_\Omega(b,r)}(v,b)}\ <\ +\infty
\]
as~$\delta\to 0$ due to Corollary~\ref{cor:near-bdry=1} (which allows one to replace~$b$ by an inner point~$b_{\varepsilon r}$ lying close enough to~$b$, cf. the proof of Proposition~\ref{prop:P-close}) and Corollary~\ref{cor:G-conv} (which provides the convergence of Green's functions); {see also~\eqref{eq:h1(b)/h2(b)def} for a discussion of the ratio of two harmonic functions $G_\Omega(0,\cdot)$ and $G_{B_\Omega(b,r)}(v,\cdot)$ at/near the (degenerate) prime end $b$.}
\end{itemize}
Passing to the limit~$\delta\to 0$ in~\eqref{eq:discrete-mart-ts} we are now able to conclude that, for each $r>0$, the (continuous, uniformly bounded) process
\begin{equation}
\label{eq:P-is-mart}
P_{\Omega\smallsetminus\gamma[0,t\wedge\tau_r]}(v)\ \text{is a martingale for each}\  v\in B_\Omega(b,\tfrac{1}{2}r).
\end{equation}

We now claim that the {real-valued} process~$\xi_{t\wedge\tau_r}$ is a continuous local semi-martingale since it can be uniquely reconstructed as a {certain} deterministic function
\[
\biggl(\frac{\Im Z_1-\xi}{|Z_1-\xi|^2}\,,\,\frac{\Im Z_2-\xi}{|Z_2-\xi|^2}\,,\,Z_1\,,\,Z_2\biggr)\ \mapsto\ \xi
\]
of continuous martingales~\eqref{x:Pdef} evaluated at two distinct points~$v_1,v_2\in B_\Omega(b,\tfrac{1}{2}r)$ and differentiable { (complex-valued)} processes~${Z_1:=}g_{t\wedge \tau_r}(\phi(v_1))$, ${Z_2:=}g_{t\wedge \tau_r}(\phi(v_2))$.

{In particular, we can apply the It\^o calculus to observables~\eqref{x:Pdef}.} Using the Loewner equation~\eqref{eq:Loewner} and It\^o's lemma, one gets the following formula:
\begin{align*}
\mathrm{d}P_{\Omega\smallsetminus\gamma[0,t\wedge\tau_r]}(v)\ &=\ -\frac{1}{\pi}\mathrm{d}\Im\frac{1}{g_{t\wedge\tau_r}(\phi(v))-\xi_{t\wedge\tau_r}}\\
&=\ -\frac{1}{\pi}\Im\biggl[\frac{\mathrm{d}\xi_{t\wedge\tau_r}}{(g_{t\wedge\tau_r}(\phi(v))-\xi_{t\wedge\tau_r})^2}+ \frac{\mathrm{d}\langle\xi,\xi\rangle_{t\wedge\tau_r}-2\mathrm{d}(t\wedge\tau_r)}{(g_{t\wedge\tau_r}(\phi(v))-\xi_{t\wedge\tau_r})^3}\biggr]
\end{align*}
(here and below we use the sign~$\mathrm{d}$ for the stochastic differential). As this process should be a martingale \emph{for each}~$v\in B_\Omega(b,\frac{1}{2}r)$, the only possibility is that
\[
\text{both processes}~\xi_{t\wedge\tau_r}~\text{and}~\langle\xi,\xi\rangle_{t\wedge\tau_r}-2\mathrm{d}(t\wedge\tau_r)~\text{are (local) martingales.}
\]
Since~$\tau_r\to+\infty$ almost surely, one concludes that~$\xi_t\mathop{=}\limits^{(d)}\sqrt{2}B_t$ by the L\'evy theorem.

\begin{remark} \label{rem:Pm-mart} The martingale property~\eqref{eq:P-is-mart} can be directly generalized to the massive setup. Namely, for each subsequential limit (in the same topologies as above) of massive LERW on~$(\Od;a^\delta,b^\delta)$ {and each $v\in B_\Omega(b,\tfrac{1}{2}r)$} the following holds:
\begin{equation}
\label{eq:Pm-is-mart}
\text{the process}\ t \mapsto P^{(m)}_{\Omega\smallsetminus\gamma[0,t\wedge\tau_r]}(v)\cdot N^{(m)}_{\Omega\smallsetminus\gamma[0,t\wedge\tau_r]}\ \text{is a martingale},
\end{equation}
{where the \emph{massive Poisson kernels}~$P^{(m)}(\cdot)$ are defined by~\eqref{eq:defPm} and the additional (random) \emph{normalization factors} $N^{(m)}$ are given by~\eqref{eq:defNm}.}
In order to prove~\eqref{eq:Pm-is-mart} one mimics the arguments given above basing upon
\begin{itemize}
\item {the convergence, as $\delta\to 0$, of massive martingale observables~\eqref{eq:defMm}
 to multiplies of massive Poisson kernels~$P^{(m)}_{\Omega\smallsetminus\gamma[0,t\wedge\tau_r]}(\cdot)$; this convergence is provided by Proposition~\ref{prop:Mconv};}

\smallskip

\item the uniform boundedness of massive observables (until time~$t\wedge\tau_r$), which follows from Corollary~\ref{cor:Zm/Z>c} and the uniform boundedness of massless ones.
\end{itemize}
We identify the law of~$\xi_t$ in the massive setup in Section~\ref{sect:computations} using~\eqref{eq:Pm-is-mart} in the same spirit as discussed above in the classical situation; see {Lemma~\ref{lem:dNmt=}.}
\end{remark}

\subsection{The density of mLERW with respect to the classical LERW}\label{sect:density} Given a discrete domain~$(\Od;a^\delta,b^\delta)$ and~$m<\delta^{-1}$, denote by~$\mathbb P_{(\Od;a^\delta,b^\delta)}[\gamma^\delta]$ and $\mathbb P^{(m)}_{(\Od;a^\delta,b^\delta)}[\gamma^\delta]$ the probabilities that a simple lattice path~$\gamma^\delta$ running from~$a^\delta$ to~$b^\delta$ inside~$\Od$ appears as a classical ($m=0$) or a massive LERW trajectory, respectively.

\begin{lemma}\label{lem:ZZ/Z(a,b)<c}
Let~$\Od$ be a simply connected discrete domain, {$a^\delta,b^\delta\in\partial\Od$ (where the boundary $\partial\Od$ of $\Od$ is understood as in~\eqref{x:paO}),}
and~$v^\delta\in\Int\Od$. Then, the following estimate holds:
\[
\frac{\Z_{\Od}(a^\delta,v^\delta)\Z_{\Od}(v^\delta,b^\delta)}{\Z_{\Od}(a^\delta,b^\delta)}\ \le\ \mathrm{const},
\]
with a universal (i.e., independent of~$\Od$, $a^\delta$, $b^\delta$, and~$v^\delta$) constant.
\end{lemma}
\begin{proof} E.g., see~\cite[Proposition~3.1]{chelkak-toolbox} which claims that the left-hand side is uniformly comparable to the probability that the random walk trajectory started at~$a^\delta$ and conditioned to exit~$\Od$ at~$b^\delta$ intersects the ball~$B(v^\delta,\frac{1}{3}\mathrm{dist}(v^\delta,\partial\Od))$.
\end{proof}
\begin{proposition} \label{prop:Zm/Z(a,b)>c}
There exists a universal constant~$c_0>0$ such that, for each discrete domain~$\Od\!\subset\!B(0,R)$, boundary points~$a^\delta,b^\delta\!\in\partial\Od$ and \mbox{$m\le\frac{1}{2}\delta^{-1}$,} one has
\begin{equation}
\label{eq:Zm/Z(a,b)>c}
{\Zm_{\Od}(a^\delta,b^\delta)}/{\Z_{\Od}(a^\delta,b^\delta)}\ \ge\ \exp(-c_0m^2R^2),
\end{equation}
where the massive random walk partition function~$\Zm_{\Od}$ is defined by~\eqref{eq:defZm}.
\end{proposition}

\begin{proof} By Jensen's inequality,
\[
\frac{\Zm_{\Od}(a^\delta,b^\delta)}{\Z_{\Od}(a^\delta,b^\delta)}\ =\
\mathbb{E}_{\mathrm{SRW}(\Od;a^\delta,b^\delta)}[(1-m^2\delta^2)^{\#\pi^\delta}]\ \geq\ (1-m^2\delta^2)^{\mathbb{E}_{\mathrm{SRW}(\Od;a^\delta,b^\delta)}[\#\pi^\delta]},
\]
where the expectation is taken over {simple random walks (SRW) $\pi^\delta$} 
started at~$a^\delta$ and conditioned to exit~$\Od$ at~$b^\delta$, whereas Lemma~\ref{lem:ZZ/Z(a,b)<c} gives
\[
\mathbb{E}_{\mathrm{SRW}(\Od;a^\delta,b^\delta)}[\#\pi^\delta]+1\ =\ \sum\nolimits_{v^\delta\in\Int\Od} \frac{\Z_{\Od}(a^\delta,v^\delta)\Z_{\Od}(v^\delta,b^\delta)}{\Z_{\Od}(a^\delta,b^\delta)}\,\le\,\mathrm{const}\cdot \delta^{-2}R^2.
\]
The desired uniform estimate~\eqref{eq:Zm/Z(a,b)>c} follows easily.
\end{proof}

\begin{corollary}\label{cor:Dm} Let~${D}^{(m)}_{(\Od;a^\delta,b^\delta)}(\gamma^\delta):= \mathbb{P}^{(m)}_{(\Od;a^\delta,b^\delta)}(\gamma^\delta)/\mathbb{P}_{(\Od;a^\delta,b^\delta)}(\gamma^\delta)$. Then,

\smallskip

\noindent (i)\phantom{i} ${D}^{(m)}_{(\Od;a^\delta,b^\delta)}(\gamma^\delta)\,\le\,\exp(c_0m^2R^2)$, for each simple path~$\gamma^\delta$ from~$a^\delta$ to~$b^\delta$ in~$\Od$;

\smallskip

\noindent (ii) $\mathbb E_{(\Omega^\delta;a^\delta,b^\delta)} \bigl[\,\log{D}^{(m)}_{(\Od;a^\delta,b^\delta)}(\gamma^\delta)\,\bigr]\ge -c_0m^2R^2$, where the expectation is taken over the classical LERW measure~$\mathbb{P}_{(\Od;a^\delta,b^\delta)}$.
\end{corollary}

\begin{proof} (i) By definition,
\[
{D}^{(m)}_{(\Od;a^\delta,b^\delta)}(\gamma^\delta)\ =\ \frac{\sum_{\pi^\delta:\mathrm{LE}(\pi^\delta)=\gamma^\delta}(\frac{1}{4}(1\!-\!m^2\delta^2))^{\#\pi^\delta}} {\sum_{\pi^\delta:\mathrm{LE}(\pi^\delta)=\gamma^\delta}(\frac{1}{4})^{\#\pi^\delta}}\,\cdot\, \frac{\Z_{\Od}(a^{\delta};b^\delta)}{\Zm_{\Od}(a^{\delta};b^\delta)},
\]
where~$\mathrm{LE}$ denotes the loop-erasure procedure applied to the simple random walk trajectory~$\pi^\delta$. The estimate~\eqref{eq:Zm/Z(a,b)>c} gives the desired uniform upper bound.

\smallskip

\noindent (ii) By Jensen's inequality and since~$\Z_{\Od}(a^\delta,b^\delta)/\Zm_{\Od}(a^\delta,b^\delta)\ge 1$, one has
\[
\mathbb E_{(\Omega^\delta;a^\delta,b^\delta)} \bigl[\,\log{D}^{(m)}_{(\Od;a^\delta,b^\delta)}(\gamma^\delta)\,\bigr]\ \ge\ \log(1\!-\!m^2\delta^2)\cdot \mathbb{E}_{\mathrm{SRW}(\Od;a^\delta,b^\delta)}[\#\pi^\delta],
\]
where the first expectation is taken with respect to the LERW measure while the second is with respect to the \emph{simple} random walk measure on the set~$S_{\Od}(a^\delta,b^\delta)$. The proof is completed in the same way as the proof of Proposition~\ref{prop:Zm/Z(a,b)>c}.
\end{proof}

Below we also need the following extension of Lemma~\ref{lem:ZZ/Z(a,b)<c} and Proposition~\ref{prop:Zm/Z(a,b)>c}.
\begin{lemma}\label{lem:ZZ/Z<cG}
Let $\Od$ be a discrete domain, $z^\delta,w^\delta\in\overline{\Omega}{}^\delta$ and~$v^\delta\in\Int\Od$. Then,
\begin{equation}
\label{eq:ZZ/Z<cG}
\frac{\Z_{\Od}(w^\delta,v^\delta)\Z_{\Od}(v^\delta,z^\delta)}{\Z_{\Od}(w^\delta,z^\delta)}\ \le\ \mathrm{const}\cdot(1+\Z_{\Od}(w^\delta,v^\delta)+\Z_{\Od}(v^\delta,z^\delta)),
\end{equation}
with a universal (i.e., independent of~$\Od$, $w^\delta$, $z^\delta$, and~$v^\delta$) constant.
\end{lemma}
\begin{proof} Denote~$d_{\Od}(v^\delta):=\dist(v^\delta,\partial\Od)$. Standard estimates imply that
\[
\Z_{\Od}(w^\delta,v^\delta)\ \le\ \const\cdot\Z_{\Od}(w^\delta,z^\delta)\ \ \text{if}\ \ |z^\delta-v^\delta|\le \tfrac{1}{3}d_{\Od}(v^\delta)~\text{and}~|z^\delta-v^\delta|\le|w^\delta-v^\delta|.
\]
(Indeed, if~$|w^\delta\!-\!v^\delta|\ge\tfrac{2}{3}d_{\Od}(v^\delta)$, then both sides are comparable due to the Harnack principle, otherwise one has~$\Z_{\Od}(w^\delta,v^\delta)\le\const\cdot \Z_{\Od}(z^\delta,v^\delta)\le\const\cdot \Z_{\Od}(z^\delta,w^\delta)$). In particular, this proves the desired estimate in the situation when~$z^\delta$ (or, similarly,~$w^\delta$) is within~$\frac{1}{3}d_{\Od}(v^\delta)$ distance from~$v^\delta$.

To handle the case when both~$w^\delta$ and~$z^\delta$ are at least~$\frac{1}{3}d_{\Od}(v^\delta)$ apart from~$v^\delta$, note that {the ratio of two positive discrete harmonic functions satisfies the maximum principle: if the inequality $H_1\le CH_2$ holds at all neighbors of a given vertex, then it also holds at this vertex since the function $H_1-CH_2$ is discrete harmonic.

Therefore,} the left-hand side of~\eqref{eq:ZZ/Z<cG} satisfies the maximum principle in both variables~$w^\delta$ and~$z^\delta$; is uniformly bounded due to Lemma~\ref{lem:ZZ/Z(a,b)<c} if both~$w^\delta,z^\delta\in\partial\Od$; and is also uniformly bounded if at least one of these two vertices is at distance~$\frac{1}{3}d_{\Od}(v^\delta)$ from~$v^\delta$ due to the argument given above.
\end{proof}

\begin{corollary}\label{cor:Zm/Z>c} There exists a universal constant~$c_0>0$ such that, for each discrete domain~$\Od\subset B(0,R)$, two vertices~$w^\delta,z^\delta\in\overline\Omega{}^\delta$ and~$m\le\frac{1}{2}\delta^{-1}$, one has
\[
{\Zm_{\Od}(w^\delta,z^\delta)}/{\Z_{\Od}(w^\delta,z^\delta)}\ \ge\ \exp(-c_0 m^2R^2).
\]
\end{corollary}

\begin{proof} The proof mimics the proof of Proposition~\ref{prop:Zm/Z(a,b)>c}. Indeed, one has
\[
\mathbb E_{\mathrm{SRW}(\Od;z^\delta,w^\delta)}[\,\#\pi^\delta\,]\le\const\ \cdot\!\!\!\!\sum_{v^\delta\in\Int\Od}\!\!(1+\Z_{\Od}(w^\delta,v^\delta)+\Z_{\Od}(v^\delta,z^\delta))\le \const\cdot \delta^{-2}R^2
\]
due to Lemma~\ref{lem:ZZ/Z<cG} and standard estimates of the discrete Green functions.
\end{proof}

\subsection{Absolute continuity of mSLE(2) with respect to SLE(2)}\label{sect:abs-cont}
As discussed in Section~\ref{sect:topologies}, the classical LERW probability measures~$\mathbb{P}_{(\Od;a^\delta,b^\delta)}$ on curves in discrete approximations~$(\Od;a^\delta,{b^\delta})$ are tight. Moreover (see Section~\ref{sect:conv_m=0}), the only possible weak limit of $\mathbb P_{(\Od;a^\delta,b^\delta)}$, as~$\delta\to 0$, is given by the SLE(2) measure on curves in~$(\Omega;a,b)$, which we denote by~$\mathbb P_{(\Omega;a,b)}$. Due to Corollary~\ref{cor:Dm}(i), the densities
\[
{D}^{(m)}_{(\Od;a^\delta,b^\delta)}(\gamma^\delta)\ =\ \mathbb P^{(m)}_{(\Od;a^\delta,b^\delta)}(\gamma^\delta)/\mathbb P_{(\Od;a^\delta,b^\delta)}(\gamma^\delta)
\]
of the massive LERW measures on curves in~$(\Od;a^\delta,b^\delta)$ with respect to the classical ones are uniformly bounded from above by~$\exp(c_0m^2R^2)$. Therefore, the measures~$\mathbb P^{(m)}_{(\Od;a^\delta,b^\delta)}$ are also tight in the topologies discussed in Section~\ref{sect:topologies}.

\begin{lemma}\label{lem:density}
(i) Each subsequential weak limit~$\mathbb P^{(m)}_{(\Omega;a,b)}$ of the massive LERW measures $\mathbb P^{(m)}_{(\Od;a^\delta,b^\delta)}$ is absolutely continuous with respect to the SLE(2) measure~$\mathbb P_{(\Omega;a,b)}$.
The Radon--Nikodym derivative ${D}^{(m)}_{(\Omega;a,b)}:=d\mathbb P^{(m)}_{(\Omega;a,b)}/d\mathbb P_{(\Omega;a,b)}$ is (almost surely) bounded from above by~$\exp(c_0m^2R^2)$, with the same constant~$c_0$ as in Corollary~\ref{cor:Dm}.

\smallskip

\noindent (ii) Moreover, one has $\mathbb E_{(\Omega;a,b)}[\,\log {D}^{(m)}_{(\Omega;a,b)}\,]\ge -c_0m^2R^2$. In~particular, the measures $\mathbb P^{(m)}_{(\Omega;a,b)}$ and~$\mathbb P_{(\Omega;a,b)}$ are mutually absolutely continuous.
\end{lemma}

\begin{proof} Denote~$C:=\exp(c_0m^2R^2)$. Both results can be easily deduced from Corollary~\ref{cor:Dm} by passing to the limit~$\delta\to 0$. As  probability measures on metrizable spaces are always regular, each Borel set~$A$ can be approximated by a compact subset~$F\subset A$. In its turn, $F$ can be approximated by its open \mbox{$\varepsilon$-neighborhood}~$F^\varepsilon$ that can be without loss of generality assumed to be a continuity set for both measures under consideration. The first claim easily follows since
\[
\mathbb P^{(m)}_{\Omega}[F^\varepsilon]\ =\ \lim_{\delta\to 0}\mathbb P^{(m)}_{\Od}[F^{\varepsilon}]\ \le\
C\cdot \lim_{\delta\to 0}\mathbb P_{\Od}[F^{\varepsilon}]\ =\ C\cdot \mathbb P_{\Omega}[F^\varepsilon]
\]
for such approximations of~$A$, here and below we write~$\Omega$ instead of~$(\Omega;a,b)$ and~$\Od$ instead of~$(\Od;a^\delta,b^\delta)$ for shortness. Therefore, $\mathbb P^{(m)}_{\Omega}[A]\le C\cdot \mathbb P_{\Omega}[A]$ for each Borel set~$A$. To prove~(ii), note that
\[
\mathbb E_{\Omega}[\,\log D^{(m)}_{\Omega}\,]\ = \inf_{A_k - \text{disjoint}\,:\,\mathbb P_{\Omega}(\cup_{k=1}^n A_k)=1} \,\biggl\{\,\sum_{k=1}^n \mathbb P_{\Omega}[A_k]\log\frac{\mathbb P^{(m)}_{\Omega}[A_k]}{\mathbb P_{\Omega}[A_k]}\,\biggr\}
\]
and approximate each~$A_k$ by~$F_k^\varepsilon$ as explained above. Provided that~$\varepsilon>0$ is small enough (depending on the choice of~$F_k$), the sets~$F_k^\varepsilon$ are still disjoint and hence
\[
\mathbb E_{\Od}[\,\log D^{(m)}_{\Od}\,]\ \le\ \sum_{k=1}^n \mathbb P_{\Od}[F_k^\varepsilon]\log\frac{\mathbb P^{(m)}_{\Od}[F_k^\varepsilon]}{\mathbb P_{\Od}[F_k^\varepsilon]}\ + (1-\mathbb P_{\Od}\bigl[\cup_{k=1}^n F_k^\varepsilon\bigr])\cdot \log C.
\]
The proof is completed by applying the uniform estimate~$\mathbb E_{\Od}[\,\log D^{(m)}_{\Od}\,]\ge-\log C$ provided  by Corollary~\ref{cor:Dm}(ii), passing to the limit~$\delta\to 0$, and then passing to the limit in the choice of approximations~$F_k^\varepsilon$ of a given disjoint collection~$A_k$.
\end{proof}

We now discuss how the law of the driving term~$\xi_t=\sqrt{2}B_t$ of SLE(2) changes when the measure~$\mathbb P_{(\Omega;a,b)}$ is replaced by~$\mathbb P^{(m)}_{(\Omega;a,b)}$. Let
\[
D^{(m)}_t\ :=\ \mathbb{E}[\,{D}^{(m)}_{(\Omega;a,b)}\,|\,\mathcal{F}_t\,],
\]
where~$\mathcal{F}_t$ denotes the (completed) canonical filtration of the Brownian motion~$B_t$. Since~${D}^{(m)}_{(\Omega;a,b)}>0$ almost surely, $D^{(m)}_t$ is a continuous martingale taking (strictly) positive values (e.g., see~\cite[p.~107]{legall-book}). Therefore (see~\cite[{ Proposition~5.8}]{legall-book}), there exists a unique continuous local martingale $L^{(m)}_t$ such that
\begin{equation}
\label{x:Ldef}
D^{(m)}_t\ =\ \exp\bigl(L^{(m)}_t\!-\! \tfrac{1}{2}\langle L^{(m)},L^{(m)} \rangle_t\bigr)
\end{equation}
and the Girsanov theorem (see~\cite[Theorem~5.8]{legall-book}) implies that 
\begin{equation}
\label{eq:Xi=}
 \xi_t=\sqrt{2}\cdot (B_t+\langle B,L^{(m)}\rangle_t)\quad \mathrm{under}\quad \mathbb{P}^{(m)}_{(\Omega;a,b)}\,.
\end{equation}

Let~$\tau_n\to\infty$ be stopping times that localize~$L^{(m)}_t$. Jensen's inequality (which can be applied due to Lemma~\ref{lem:density}(i)) and Lemma~\ref{lem:density}(ii) imply that
\begin{equation}
\label{eq:LL<infty}
\mathbb E[\tfrac{1}{2}\langle L^{(m)},L^{(m)}\rangle_\infty]\ =\lim_{\tau_n\to\infty}\mathbb E[-\log D^{(m)}_{\tau_n}]\ \le\ \mathbb E [-\log D^{(m)}_{(\Omega;a,b)}]\ \le\ c_0m^2R^2.
\end{equation}
In particular,~$\langle L^{(m)},L^{(m)}\rangle_\infty<+\infty$ a.\,s. In fact, \emph{a posteriori} one can deduce from Theorem~\ref{theorem} that \mbox{$\langle L^{(m)},L^{(m)}\rangle_\infty\le \mathrm{const}(m,R)<+\infty$} a.\,s. (see~Remark~\ref{rem:int-lambda2}). Note however that we need some \emph{a priori} information on~$L^{(m)}$ to prove this theorem.

\begin{remark} \label{rem:Nm-not-mart}
By definition, the process~$(D^{(m)}_t)^{-1}$ is a local martingale under~$\mathbb P^{(m)}_{(\Omega;a,b)}$. Assume that, for an adapted process~$\lambda_t$, one has
\begin{equation}
\label{x:lambda}
\mathrm{d}(D^{(m)}_t)^{-1}\ =\ -\sqrt{2}\lambda_t\cdot (D^{(m)}_t)^{-1}\cdot \mathrm{d}B_t \quad \mathrm{under}\quad \mathbb{P}^{(m)}_{(\Omega;a,b)}\,.
\end{equation}
Due to~\eqref{x:Ldef}, this implies that the martingale part of the process~$L_t$ (which is a semi-martingale under~$\mathbb P^{(m)}_{(\Omega;a,b)}$) is~$\sqrt{2}\lambda_t\mathrm{d}B_t$ and hence
\[
\mathrm{d}\xi_t\ =\ \sqrt{2}\mathrm{d}B_t+2\lambda_t\mathrm{d}t \quad \mathrm{under}\quad \mathbb{P}^{(m)}_{(\Omega;a,b)}.
\]
Therefore, in order to find the law of~$\xi_t$ it is enough to identify~$\lambda_t$ in~\eqref{x:lambda}. It is worth noting that in the \emph{massive} setup
\[
(D^{(m)}_t)^{-1}\ \neq\ \lim_{\delta\to 0} \bigl(\,\Z_{\Od_t}(a^\delta_t,b^\delta)/\Zm_{\Od_t}(a^\delta_t,b^\delta)\,\bigr)\ =:\ N^{(m)}_t,
\]
a standard identity, e.g., in the \emph{multiple} SLE context. The reason is that the total mass of massive RW loops attached to the tip~$a_t^\delta$ is strictly smaller than the mass of the critical ones. Because of that, the process~$N^{(m)}_t$ actually has a negative drift (which can be computed explicitly, see~\eqref{eq:dN=}) and one cannot easily deduce Theorem~\ref{theorem} relying only upon the analysis of this process; cf. Remark~\ref{rem:Nm-vs-Nb}.
\end{remark}

\section{Convergence of martingale observables}\label{sect:convergence}

\setcounter{equation}{0}

\subsection{Convergence of discrete harmonic functions} In this section we recall two useful results from~\cite{chelkak-smirnov-11}: convergence of the discrete Green functions $\Z_{\Od}(u^\delta,v^\delta)$ and of the discrete Poisson kernels~$\Z_{\Od}(a^\delta,u^\delta)/Z_{\Od}(a^\delta,v^\delta)$ as~$\contOd\to\Omega$, where~$u,v$ are inner points and~$a$ is a boundary point (more accurately, a prime end) of~$\Omega$. Recall that we denote by~$\contOd$ the polygonal representation of a discrete domain~$\Omega^\delta$. 
\begin{definition}\label{def:r-inside}
Let $\Omega\subset\mathbb C$ be a simply connected bounded domain and~$r>0$. We say that points $u,v \in \Omega$ are \emph{jointly $r$-inside~$\Omega$} if 
they can be connected by a path~$L_{uv}\subset\Omega$ such that $\dist(L_{uv},\partial\Omega) > r$. In other words, $u$ and $v$ belong to the same connected component of the $r$-interior of $\Omega$.
\end{definition}

{In what follows, we assume that all domains under considerations are \emph{uniformly} bounded. This assumption is mostly technical; in particular, it slightly simplify the discussion of subsequential limits of $\widehat\Omega{}^\delta$ in the Carath\'eodory topology, which is useful when speaking about uniform (with respect to $\widehat\Omega{}^\delta$) estimates; cf.~\cite{chelkak-smirnov-11}.}

\begin{proposition}\label{prop:G-close} Let~$0<r<R$ be fixed. There exists a function~$\varepsilon(\delta)=\varepsilon(\delta,r,R)$, defined for small enough~$\delta\le\delta_0(r,R)$, such that~$\varepsilon(\delta)\to 0$ as~$\delta\to 0$ and that the following is fulfilled for all simply connected discrete domains~$\contOd\subset B(0,R)$ and all pairs of points~$u^\delta,v^\delta$ lying jointly r-inside~$\contOd$ and such that~$|u^\delta-v^\delta|\ge r$:
\begin{equation}
\label{eq:G-close}
|\Z_{\Od}(u^\delta,v^\delta)-G_{\widehat{\Omega}^\delta}(u^\delta,v^\delta)|\ \le\ \varepsilon(\delta).
\end{equation}
\end{proposition}
\begin{proof} This follows from (a more general in several aspects) uniform convergence result provided by~\cite[Corollary~3.11]{chelkak-smirnov-11} 
and the convergence of the discrete \emph{full-plane} Green function on the rescaled grid $\delta\mathbb{Z}$ to~$-\frac{1}{2\pi}\log|u^\delta-v^\delta|$ (up to a constant) for~$r\le |u^\delta-v^\delta|\le 2R$ and~$\delta\to 0$, the latter being a standard fact of the discrete potential theory on the square grid.
\end{proof}

\begin{corollary} \label{cor:G-conv} Let~$\Omega\subset B(0,R)$ be a simply connected planar domain and $u,v\in\Omega$ be two distinct points of~$\Omega$. Assume that discrete domains~$\contOd\subset B(0,R)$ approximate $\Omega$ (in the Cara\-th\'eodory topology with respect to~$u$ or~$v$) as~$\delta\to 0$. Then,
\begin{equation}
\label{eq:G-conv}
\Z_{\Od}(u^\delta,v^\delta)\ \to\ G_\Omega(u,v)\quad \text{as}\ \ \delta\to 0.
\end{equation}
Moreover, for each~$r>0$ this convergence is uniform provided that~$u$ and~$v$ are jointly~$r$-inside~$\Omega$ and~$|u-v|\ge r$.
\end{corollary}
\begin{proof} Let~$L_{uv}\subset\Omega$ be a path connecting~$u$ and~$v$ inside~$\Omega$ and~$r:=\frac{1}{2}\dist(L_{uv},\partial\Omega)$. It follows from the Carath\'eodory convergence of~$\contOd$ to~$\Omega$ that~$u^\delta$ and~$v^\delta$ are jointly \mbox{$r$-inside} of~$\Omega^\delta$ provided that~$\delta$ is small enough. Since (the continuous) Green function is conformally invariant,~$G_{\contOd}(u^\delta,v^\delta)\to G_\Omega(u,v)$ as~$\delta\to 0$ uniformly for such~$u$ and~$v$ and thus the claim trivially follows from~\eqref{eq:G-close}.
\end{proof}
\begin{remark} In Section~\ref{sect:mGreen} we prove an analogue of~\eqref{eq:G-conv} in the massive setup along the lines of~\cite{chelkak-smirnov-11} though do not discuss an analogue of~\eqref{eq:G-close}. Note that in~\cite{chelkak-smirnov-11} the uniform estimate~\eqref{eq:G-close} is actually \emph{deduced from~\eqref{eq:G-conv}} by compactness arguments; cf. the proofs of Proposition~\ref{prop:P-close} and Corollary~\ref{cor:P-conv} discussed below.
\end{remark}

\begin{proposition}\label{prop:P-close}
Let~$0<r<R$ be fixed. There exists a function~$\varepsilon(\delta)=\varepsilon(\delta,r,R)$, defined for small enough~$\delta\le\delta_0(r,R)$, such that~$\varepsilon(\delta)\to 0$ as~$\delta\to 0$ and that the following is fulfilled for all simply connected discrete domains~$\contOd\subset B(0,R)$, all boundary points~$a^\delta$, and all inner points~$u^\delta,v^\delta\in\Omega^\delta$ lying jointly r-inside~$\contOd$:
\begin{equation}
\label{eq:P-close}
\biggl|\,\frac{\Z_{\Od}(a^\delta,u^\delta)}{\Z_{\Od}(a^\delta,v^\delta)} \,-\, \frac{P_{\widehat{\Omega}^\delta}(a^\delta,u^\delta)}{P_{\widehat{\Omega}^\delta}(a^\delta,v^\delta)}\,\biggr|\ \le\ \varepsilon(\delta),
\end{equation}
where~$P_{\contOd}(a^\delta,\cdot)$ denotes the Poisson kernel in the polygonal representation~$\contOd$ with mass at the point~$a^\delta\in\partial\contOd$, note that its normalization is irrelevant for~\eqref{eq:P-close}.
\end{proposition}

\begin{proof} This result is provided (again, in a stronger form) by~\cite[Theorem~3.13]{chelkak-smirnov-11}. For completeness of the exposition we sketch the key ingredients of this proof, which goes by contradiction. If the uniform estimate~\eqref{eq:P-close} was wrong, it would fail (for a fixed~$\varepsilon_0>0$) along a sequence of configurations~$(\Omega^\delta;a^\delta,u^\delta,v^\delta)$ with~$\delta\to0$. As the set of all simply connected domains~$\Lambda$ satisfying~$B(u,r)\subset \Lambda\subset B(0,R)$ is \emph{compact} in the Carath\'eodory topology, we could pass to a subsequence and assume that~$(\contOd;a^\delta,u^\delta,v^\delta)\to (\Omega;a,u,v)$ as~$\delta\to 0$ {in the Carath\'eodory sense}, with~$u$ and~$v$ being jointly~$r$-inside~$\Omega$. The {ratio of} Poisson kernels~$P_\Lambda(a,u)/P_\Lambda(a,v)$ is conformally invariant and so is stable under this convergence. Thus, it is enough to prove that
\begin{equation}
\label{x:P-conv}
\frac{\Z_{\Od}(a^\delta,u^\delta)}{\Z_{\Od}(a^\delta,v^\delta)}\ \to\ \frac{P_\Omega(a,u)}{P_\Omega(a,v)}\quad\text{as}\ \ (\contOd;a^\delta,u^\delta,v^\delta)\ \mathop{\longrightarrow}\limits^{\mathrm{Cara}}\ (\Omega;a,u,v)
\end{equation}
in order to obtain a contradiction, where~$u,v\in\Omega$ and~$a$ is a \emph{prime end} of~$\Omega$.

\smallskip

Let~$d>0$ be small enough and let a point~$a_d$ be chosen so that the circle $\partial B(a_d,\frac{1}{2}d)$ separates the prime end~$a$ from $u$ and~$v$ in~$\Omega$. Since~$(\Omega^\delta;a^\delta)$ converges to~$(\Omega;a)$, the circle~$\partial B(a_d,d)$ then separates~$a^\delta$ from $u^\delta$ and~$v^\delta$ in~$\contOd$, for all sufficiently small~$\delta$. Let~$L^\delta_d\subset \partial B(a_d,d)$ denote the arc separating~$u^\delta$ and~$v^\delta$ from~$a^\delta$ and all the other arcs forming the set~$\partial B(a_d,d)\cap\widehat{\Omega}^\delta$, in other words this is the first arc of~$\partial B(a_d,d)\cap\widehat{\Omega}^\delta$ to cross for a path running from, say,~$u^\delta$ to~$a^\delta$; see~\cite[Fig.~4]{chelkak-smirnov-11}.

\smallskip

Denote by~$\Omega^\delta_{3d}$ the connected component of~$\Omega^\delta\smallsetminus B(a_d,3d)$ that contains~$v^\delta$. The key argument of the proof is the following uniform (for small enough~$\delta$) estimate:
\begin{equation}
\label{eq:upper-bound-O3r}
\max_{u^\delta\in\Omega^\delta_{3d}}\frac{\Z_{\Od}(a^\delta,u^\delta)}{\Z_{\Od}(a^\delta,v^\delta)}\ \le\ C(3d;\Omega,a).
\end{equation}
We refer the reader to~\cite[pp. 26--27]{chelkak-smirnov-11} for the proof of this statement which is based on the fact that the discrete harmonic measure~$\omega^\delta(v^\delta;K^\delta_{3d};\Omega^\delta_d)$ of each path~$K_{3d}^\delta$ started in~$\Omega^\delta_{3d}$ and running to~$L^\delta_d$ is uniformly bounded from below due to~\cite[Theorem~3.12]{chelkak-smirnov-11} and~\cite[Lemma~3.14]{chelkak-smirnov-11}; note that~$u^\delta$ is \emph{not} assumed to be located in the~$r$-interior of~$\Omega^\delta$ in~\eqref{eq:upper-bound-O3r}.

\smallskip

The proof can be now completed in a standard way. The (uniform in~$\delta$) weak-Beurling estimate (see Lemma~\ref{lem:weak Beurling}) allows one to improve the uniform bound~\eqref{eq:upper-bound-O3r} near the boundary of~$\Omega^\delta$:
\[
\frac{\Z_{\Od}(a^\delta,u^\delta)}{\Z_{\Od}(a^\delta,v^\delta)}\ \le\ \const\cdot (\dist(u^\delta,\partial\Omega^\delta)/d)^\beta\cdot C(3d;\Omega,a)\quad \text{for}\ \ u^\delta\in\Od_{4d}.
\]
Since uniformly bounded discrete harmonic functions are also equicontinuous (cf. Lemma~\ref{lem:regularity}), one can pass to a subsequence once again to get the (uniform on compact subsets) convergence
\[
\frac{\Z_{\Od}(a^\delta,u^\delta)}{\Z_{\Od}(a^\delta,v^\delta)}\ \to\ \ h(u),\quad u\in\bigcup\nolimits_{d>0}\Omega_{4d}=\Omega.
\]
Each subsequential limit~$h$ is a positive harmonic function in~$\Omega$ normalized so that~$h(v)=1$ and satisfies, for each~$d>0$, the same estimate
\[
h(u)\ \le\ \const\cdot (\dist(u,\partial\Omega)/d)^\beta\cdot C(3d;\Omega,a)\quad \text{for}\ \ u\in\Omega_{4d}.
\]
Thus,~$h$ has Dirichlet boundary conditions, except at the prime end~$a$. These properties characterize the Poisson kernel~$h(u)=P_\Omega(a,u)/P_\Omega(a,v)$ uniquely.
\end{proof}

\begin{corollary} \label{cor:P-conv}
Let~$\Omega\subset B(0,R)$ be a simply connected planar domain,~$a\in\partial\Omega$ be its prime end, and $u,v\in\Omega$ be two, not necessarily distinct, inner points. Assume that discrete domains~$\contOd\subset B(0,R)$ with marked boundary points~$a^\delta\in\partial\Omega^\delta$ approximate $(\Omega;a)$ in the Cara\-th\'eodory topology with respect to~$u$ or~$v$. Then,
\begin{equation}
\label{eq:P-conv}
\frac{\Z_{\Od}(a^\delta,u^\delta)}{\Z_{\Od}(a^\delta,v^\delta)}\ \to\ \frac{P_\Omega(a,u)}{P_\Omega(a,v)}\quad \text{as}\ \ \delta\to 0.
\end{equation}
Moreover, for each~$r>0$ this convergence is uniform if~$u,v$ are jointly~$r$-inside~$\Omega$.
\end{corollary}

\begin{proof} For a fixed pair $u,v$ of points of~$\Omega$, this result is given by~\eqref{x:P-conv} and is a key step of the proof of Proposition~\ref{prop:P-close}. The fact that the convergence is uniform provided that~$u$ and~$v$ are jointly $r$-inside~$\Omega$ can be, for instance, deduced from~\eqref{eq:P-close} and the conformal invariance of the Poisson kernel. Indeed, the Carath\'eodory convergence of~$(\contOd;a^\delta)$ to~$(\Omega;a)$ implies that ${P_{\contOd}(a,u)}/{P_{\contOd}(a,v)} \to {P_\Omega(a,u)}/{P_\Omega(a,v)}$ as~$\delta\to 0$, uniformly for such~$u$ and~$v$.
\end{proof}

\subsection{Boundary behavior of discrete harmonic functions} \label{sect:boundarybehavior} Since we work in the chordal setup, in order to prove the convergence of the martingale observables~\eqref{eq:defMm} we need convergence results for (both classical and massive) Poisson kernels normalized at the boundary. To make the exposition self-contained and accessible to readers who are not familiar with the classical potential theory in 2D, we start this section with a remark on the boundary behavior of continuous harmonic functions defined in a vicinity~$B_\Omega(b,r)\subset\Omega$ of its degenerate prime end~$b$ and satisfying the \emph{zero} Dirichlet boundary conditions on~$\partial B_\Omega(b,r)\cap \partial\Omega$.

Given two such (positive) functions~$h_1,h_2:B_\Omega(b,r)\to\mathbb R_+$, we claim that their \emph{ratio}~$h_1/h_2$ is always continuous at~$b$ and we slightly abuse the notation by writing
\begin{equation}
\label{eq:h1(b)/h2(b)def}
\frac{h_1(b)}{h_2(b)}\ :=\ \lim\nolimits_{\rho_\Omega(b,z)\to 0}\frac{h_1(z)}{h_2(z)},
\end{equation}
Indeed, let~$\phi:B_\Omega(b,r)\to\mathbb H$ be a conformal uniformization of~$B_\Omega(b,r)$ onto the upper half-plane~$\mathbb H$ such that~$\phi(b)=0$. Both functions~$h_{1,2}\circ\phi^{-1}$ are harmonic in {the upper half-plane~$\mathbb H$ and have Dirichlet boundary values near $0$. By the Schwarz reflection principle, these functions} must behave like~$c_{1,2}\Im z+O(|z|^2)$ as~$z\to 0$, which implies the existence of the limit~$c_1/c_2$ in~\eqref{eq:h1(b)/h2(b)def}. Below we prove a similar statement in discrete, \emph{uniformly over all possible shapes} of discrete domains~$\Omega^\delta$ near~$b$. To do this, we need additional notation.

Let~$\Od$ be a simply connected discrete domain, $o\in\contOd$, $b\in\partial{\Od}$, and~\mbox{$r>2\delta$} be such that~$o\not\in B_{\widehat{\Omega}^\delta}(b,r)$. {Consider a collection of arcs $\partial B(b,r)\cap \contOd$ and denote by $S_o(b,r)$ one of these arcs that separates $o$ from $b$ in $\contOd$; if there are several such arcs, then we take the closest to $b$ among them as $S_o(b,r)$. (More precisely, we require that $S_o(b,r)$ separates $b$ from all the other arcs from this sub-collection.)}
Let~$\Omega^\delta_o(b,r)$ be the connected component of~$\Omega^\delta\smallsetminus B(b,r)$ that contains the point~$o$. Further, let~$S_o^\delta(b,r^+),S_o^\delta(b,r^-)\subset\Od$ be the sets of vertices that are adjacent to the arc~$S_o(b,r)$ from outside and from inside, respectively; see Fig.~\ref{fig:Cr}.

\begin{figure}
	\includegraphics[width=0.8\textwidth]{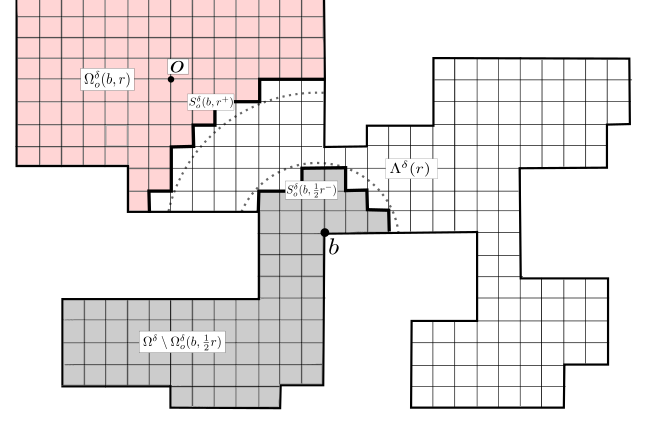}
\caption{Notation used in Lemma~\ref{lemma:-+} and Corollary~\ref{cor:near-bdry=1}.\label{fig:Cr}}
\end{figure}

\begin{lemma}\label{lemma:-+} There exists a universal constant~$k<1$ such that the following is fulfilled. In the setup described above, for each pair of positive discrete harmonic functions~$H_1,H_2:\Od\to\mathbb R_+$ satisfying the Dirichlet boundary conditions on \mbox{$\partial\Od\smallsetminus\partial \Od_o(b,r)$}, one has
\begin{align*}
\max\nolimits_{u,v\in \Od\smallsetminus \Od_o(b,\frac{1}{2}r)}\ &\biggl|\frac{H_1(u)H_2(v)-H_1(v)H_2(u)}{H_1(u)H_2(v)+H_1(v)H_2(u)}\biggr| \\ & \le\ k\cdot \max\nolimits_{x,y\in S^\delta_o(b,r^+)}\ \biggl|\frac{H_1(x)H_2(y)-H_1(y)H_2(x)}{H_1(x)H_2(y)+H_1(y)H_2(x)}\biggr|\,.
\end{align*}
\end{lemma}

\begin{proof} For shortness, denote~$\Theta^\delta(r):=\Od\setminus \Od_o(b,r)$ and $\Lambda^\delta(r):=\Theta^\delta(r)\smallsetminus\Theta^\delta(\frac{1}{2}r)$; see Fig.~\ref{fig:Cr}.
Given a discrete harmonic function \mbox{$H:\Od\to \mathbb{R}_+$} satisfying the Dirichlet boundary conditions on \mbox{$\partial\Od\smallsetminus\partial \Od_o(b,r)$} and a point~$u\in \Od\smallsetminus \Od_o(b,\frac{1}{2}r)$, one can write
\begin{align*}
H(u)\ =\sum_{x\in S^\delta_o(b,r^+)} \Z_{\Theta^\delta(r)}(u,x)H(x)\ =\!\!\!\sum_{\substack{x\in S^\delta_o(b,r^+)\\u'\in S^\delta_o(b,\frac{1}{2}r^-)}} \Z_{\Theta^\delta(r)}(u,u')\Z_{\Lambda^\delta(r)}(u',x)H(x),
\end{align*}
where~$u'$ stands for the last point in~$\Theta^\delta(\frac{1}{2}r)$ visited by a random walk trajectory running from~$u$ to~$x$.
Applying this identity four times (for both functions~$H_1,H_2$ as well as for both points~$u,v$) and rearranging terms one sees that
\begin{align*}
H_1(u)H_2(v)\mp H_1(v)H_2(u) & \\
=\ \frac{1}{2}\sum_{\substack{x,y\in S^\delta_o(b,r^+)\\u',v'\in S^\delta_o(b,\frac{1}{2}r^-)}} & \Z_{\Theta^\delta(r)}(u,u')\Z_{\Theta^\delta(r)}(v,v') \\[-12pt]  & \ \times (\Z_{\Lambda^\delta(r)}(u',x)\Z_{\Lambda^\delta(r)}(v',y)\mp\Z_{\Lambda^\delta(r)}(u',y)\Z_{\Lambda^\delta(r)}(v',x))\\  & \ \times (H_1(x)H_2(y)\mp H_2(x)H_1(y)).
\end{align*}
{Let $M:=\max_{x,y\in S_o^\delta(b,r^+)}|H_1(x)H_2(y)-H_2(x)H_1(y)|/(H_1(x)H_2(y)+H_2(x)H_1(y))$. Therefore, in order to derive the desired estimate
\[
|H_1(u)H_2(v)- H_1(v)H_2(u)|\ \le\ kM\cdot (H_1(u)H_2(v)+ H_1(v)H_2(u)),
\]
it is enough to prove that (uniformly in all the parameters involved)}
\begin{equation}
\label{x:ZZ-ZZ/ZZ+ZZ}
\biggl|\frac{\Z_{\Lambda^\delta(r)}(u',x)\Z_{\Lambda^\delta(r)}(v',y)-\Z_{\Lambda^\delta(r)}(u',y)\Z_{\Lambda^\delta(r)}(v',x)}{\Z_{\Lambda^\delta(r)}(u',x)\Z_{\Lambda^\delta(r)}(v',y)+\Z_{\Lambda^\delta(r)}(u',y)\Z_{\Lambda^\delta(r)}(v',x)}\biggr|\ \le k\,.
\end{equation}
By construction,~$\Lambda^\delta(r)$ is a simply connected domain. Without loss of generality, assume that the boundary points~$u',v',y,x$ of~$\Lambda^\delta(r)$ are listed in the counterclockwise order. Then,~\eqref{x:ZZ-ZZ/ZZ+ZZ} is equivalent to the following uniform lower bound for the \emph{discrete cross-ratio} of the quadrilateral $(\Lambda^\delta(r);u',v',y,x)$:
\[
\mathrm{X}_{\Lambda^\delta(r)}(u',v';y,x)\ :=\ \biggl[\frac{Z_{\Lambda^\delta(r)}(u',y)Z_{\Lambda^\delta(r)}(v',x)}{Z_{\Lambda^\delta(r)}(u',x)Z_{\Lambda^\delta(r)}(v',y)}\biggr]^{1/2}\ge\ \biggl[\frac{1-k}{1+k}\biggr]^{1/2}.
\]
Due to \cite[Proposition~4.5]{chelkak-toolbox} and \cite[Theorem~7.1]{chelkak-toolbox}, this estimate (with some universal constant~$k<1$) follows from the following uniform lower bound on the \emph{discrete extremal length} (aka effective resistance) between the arcs~$[u'v']$ and~$[xy]$ in~$\Lambda^\delta(r)$:
\begin{align*}
\mathrm{L}_{\Lambda^\delta(r)}([u'v']_{\Lambda^\delta(r)};[xy]_{\Lambda^\delta(r)})\ &\ge\ \mathrm{L}_{\Lambda^\delta(r)}(S^\delta_o(b,\tfrac{1}{2}r^-),S^\delta_o(b,r^+))\\ &\ge\ \const\cdot \tfrac{1}{2\pi}\log 2\ >\ 0,
\end{align*}
which holds true since the discrete and the continuous extremal lengths are uniformly comparable to each other (e.g., see \cite[Proposition~6.2]{chelkak-toolbox}) and one can replace the quadrilateral~$(\Lambda^\delta(r);u',v',x,y)$ by the annulus $B(b,r)\smallsetminus B(b,\frac{1}{2}r)$ using monotonicity properties of the extremal length.
\end{proof}

\begin{corollary} \label{cor:near-bdry=1}
In the same setup, let \mbox{$q\in\mathbb N$} and \mbox{$r>2^q\delta$} be such that~$o\not\in B_{\contOd}(b,r)$. Let~$H_1,H_2:\Od\to\mathbb R_+$ be positive discrete harmonic functions satisfying the Dirichlet boundary conditions on~$\partial \Od\setminus \partial\Od_o(b,r)$. Then, one has
\[
\max\nolimits_{u,v\in \Od\setminus \Od_o(b,2^{-q} r)}\frac{H_1(u)/H_2(u)}{H_1(v)/H_2(v)}\ \le\ \frac{1+k^q}{1-k^q}\,,
\]
with the same universal constant~$k<1$ as in Lemma~\ref{lemma:-+}.
\end{corollary}
\begin{proof} This estimate follows easily by iterating $q$ times the result of Lemma~\ref{lemma:-+}, (note that the ratio inside the absolute value is always less than $1$), which gives
$|H_1(u)H_2(v)-H_2(u)H_1(v)|\le k^q\cdot (H_1(u)H_2(v)+H_2(u)H_1(v))$.
\end{proof}

\subsection{Convergence of the massive Green function}\label{sect:mGreen} In this section we prove an analogue of the uniform convergence~\eqref{eq:G-conv} for massive Green functions~$\Zm_{\Od}(u^\delta,v^\delta)$.
To prove this result, Proposition~\ref{prop:Gm-conv}, we need several preliminary facts.

\begin{lemma}\label{lem:crossing of annuli}
Let $(X_n)_{n\in\mathbb N}$ be a simple random walk with killing rate $m^2\delta^2$ on $\delta\mathbb{Z}^2$. For an annulus $A=A(v_0,r_1,r_2)$, denote by $E(A)$ the event that $X_n$, started at $v \in A\cap \delta\mathbb{Z}^2$, makes a non-trivial loop around $v_0$ before exiting~$A$, that is, there exists $0 \leq s<k<\tau_{\mathbb{C}\smallsetminus A}$ such that~$X_s=X_k$ and $X|_{[s,k]}$ is not null-homotopic in~$A$. There exists a universal constant such that one has
\[
\mathbb{P}^{(m)}_v[\,E(A(v_0,r,2r))\,]\ \geq\ \const\ >\ 0
\]
for all~$\delta\le r\le m^{-1}$ and all~$v\in\delta\mathbb Z^2$ such that~$\frac{3}{2}r-\delta\le |v-v_0|\le \frac{3}{2}r+\delta$.
\end{lemma}

\begin{proof} The desired event can be easily constructed from a few events of a type that a random walk started at the center~$u$ of a rectangle~$[u-\frac{1}{4}r,u+\frac{1}{4}r]\times[u-\frac{1}{8}r,u+\frac{1}{8}r]$ exists it through a prescribed side \emph{not dying along the way}. As we require that the killing rate~$m^2\delta^2$ is scaled accordingly to the mesh size and that~$r\le m^{-1}$, standard estimates imply that the probability of each of these events is uniformly bounded from below by a universal constant, independent of~$\delta$ and~$r$.
\end{proof}

Given~$m>0$, we say that a function~$H$ is \emph{massive discrete harmonic} at a vertex~$v\in\delta\mathbb Z^2$ if
\begin{equation}
\label{eq:mharm-def}
H(v)\ =\ \frac{1-m^2\delta^2}{4}\sum\nolimits_{v_1\in\delta\mathbb Z^2:v_1\sim v} H(v_1).
\end{equation}
Trivially, if $H$ is positive, then it satisfies the maximum principle:~$H(v)$ cannot be bigger than all four values~$H(v_1)$ at~$v_1\sim v$. Using Lemma~\ref{lem:crossing of annuli} one can easily prove {a priori regularity of} massive discrete harmonic functions on~$\delta\mathbb Z^2$.

\begin{lemma} \label{lem:regularity} There exists universal constants~$C,\beta>0$ such that the following holds: for each
	positive massive discrete harmonic function~$H$ defined in the disc $B(v_0,2r)\cap \delta\mathbb Z^2$ with~$r\le m^{-1}$ and for each~$v_1,v_2\in B(v_0,r)\cap\delta\mathbb Z^2$ one has
\[
|H(v_2)-H(v_1)|\ \leq\ C\cdot (|v_2-v_1|/r)^\beta\cdot \max\nolimits_{v\in B(v_0,2r)\cap \delta\mathbb Z^2} H(v).
\]
 \end{lemma}
\begin{proof} Without loss of generality, assume that~$|v_2-v_1|\le \frac{1}{4}r$. The maximum principle yields the existence of a path $\gamma$ connecting $v_2$ to the boundary of~$B(v_0,2r)$ such that the values of~$H$ along~$\gamma$ are larger than $H(v_2)$. Consider a family of concentric annuli
\[
A_k:=A(v_1,2^k|v_2-v_1|,2^{k+1}|v_2-v_1|),\quad k=0,\dots,\lfloor\tfrac{1}{2}\log_2(r/|v_2-v_1|)\rfloor.
\]
Due to Lemma~\ref{lem:crossing of annuli}, for each~$k$ the probability that the random walk with killing rate~$m^2\delta^2$ started from $v_1$ is killed or does not hit~$\gamma$ while crossing~$A_k$ is uniformly bounded away from~$1$. At the same time, standard estimates imply that the probability that this random walk is killed before crossing all~$A_k$ is uniformly bounded from above by $\const\cdot m^2r|v_2-v_1|\le\const\cdot |v_2-v_1|/r$. Hence, the probability that this random walk hits~$\gamma$ before dying or exiting~$B(v_0,2r)$ is at least~$1-C(|v_2-v_1|/r)^\beta$. Therefore,
\[
H(v_1)\ \ge\ [1-C(|v_2-v_1|/r)^\beta]\cdot H(v_2),
\]
with universal constants~$C,\beta>0$.
\end{proof}

We also need the so-called weak-Beurling estimate which applies to both discrete massive harmonic and usual ($m=0$) discrete harmonic functions.

\begin{lemma} \label{lem:weak Beurling} Let~$\Od\subset\delta\mathbb Z^2$ be a simply connected discrete domain,~$c^\delta\in\partial\Od$ be a boundary point, and~$r\le m^{-1}$. Let~$H$ be discrete massive harmonic function defined in the~$r$-vicinity~$B_{\Od}(c,r)$ of~$c$ in~$\Od$ and let~$H$ satisfy the Dirichlet boundary conditions on~$\partial B_{\Od}(c,r)\cap\partial\Od$. There exist universal constants~$C,\beta>0$ such that one has
\[
|H(v)|\ \leq\ C\cdot (\rho_{\Od}(c,v)/r)^\beta\cdot \max\nolimits_{u\in B_{\Od}(c,r)} |H(u)|
\]
for all~$v\in B_{\Od}(c,r)$, where~$\rho_{\Od}(c,v)$ and~$B_{\Od}(c,r)$ are defined by~\eqref{eq:rhoB-def}.
\end{lemma}
\begin{proof} The proof is similar to the proof of Lemma~\ref{lem:regularity}: the simple random walk with killing rate~$m^2\delta^2$ started at~$v$ hits~$\partial\Od$ or dies before reaching $\partial B_{\Od}(c,r)\smallsetminus \partial\Od$ with probability at least $1-C\cdot (\rho_{\Od}(c,v)/r)^\beta$.
\end{proof}
We are now ready to prove an analogue of Proposition~\ref{prop:G-close} for massive Green functions. Given a simply connected domain~$\Lambda\subset\mathbb C$ we denote by~$G^{(m)}_\Lambda(u,v)$ the integral kernel of the operator~$(-\Delta_\Lambda+m^2)^{-1}$, where~$\Delta_\Lambda$ stands for the Laplacian in~$\Lambda$ with Dirichlet boundary conditions. In other words, the \emph{massive Green function}~$G^{(m)}_\Lambda(u,\cdot)$ is the unique solution to the equation~$(-\Delta+m^2)G^{(m)}_\Lambda(u,\cdot)=\delta_u(\cdot)$, understood in the sense of distributions, with Dirichlet boundary conditions at~$\partial\Lambda$.

\begin{proposition}\label{prop:Gm-conv} Let~$\Omega\subset B(0,R)$ be a simply connected planar domain and $u,v\in\Omega$ be two distinct points of~$\Omega$. Assume that discrete domains~$\Omega^\delta\subset B(0,R)$ approximate $\Omega$ (in the Cara\-th\'eodory topology with respect to~$u$ or~$v$). Then,
\begin{equation}
\label{eq:Gm-conv}
\Zm_{\Od}(u^\delta,v^\delta)\ \to\ G^{(m)}_\Omega(u,v)\quad \text{as}\ \ \delta\to 0.
\end{equation}
Moreover, for each~$r>0$ this convergence is uniform provided that~$u$ and~$v$ are jointly~$r$-inside~$\Omega$ and~$|u-v|\ge r$.
\end{proposition}

\begin{proof} The functions~${H^\delta(\cdot):=}\Zm_{\Od}(u^\delta,\cdot)$ are uniformly (in~$\delta$) bounded on compact subsets of~$\Omega\smallsetminus\{u\}$ as
\begin{equation}
\label{x:Zm-upper-bound}
0\ \le\ \Zm_{\Od}(u^\delta,v^\delta)\ \le\ \Z_{\Od}(u^\delta,v^\delta)\ 
\le\ \tfrac{1}{2\pi}(\log R - \log |u^\delta\!-\!v^\delta|)+O(1).
\end{equation}
Moreover, Lemma~\ref{lem:regularity} implies that these functions are also equicontinuous and hence one can find a subsequential limit~${ h}:\Omega\smallsetminus\{u\}\to\mathbb R_+$ such that
\[
{H^\delta(\cdot)\ \to\ h(\cdot)}\quad\text{as}\ \ \delta=\delta_k\to 0,
\]
uniformly on compact subsets of~$\Omega\smallsetminus\{u\}$. Furthermore, it follows from Lemma~\ref{lem:weak Beurling} that {the function~$h(\cdot)$} has Dirichlet boundary conditions everywhere at~$\partial\Omega$.

It remains to check that~$[(-\Delta+m^2)h](\cdot)=\delta_u(\cdot)$ in the sense of distributions. Let~$\phi\in C_0^\infty(\Omega)$ be a smooth function such that~$\mathrm{supp}\phi\subset\Omega$ and hence~$\mathrm{supp}\phi\subset\contOd$ provided that~$\delta$ is small enough. For~$v^\delta\in\Int\Od$, denote
\[
[\Delta^\delta \phi](v^\delta)\ :=\ \frac{1}{4\delta^2}\sum\nolimits_{v^\delta_1\in\Od:v^\delta_1\sim v}(\phi(v^\delta_1)-\phi(v^\delta)).
\]
{Recall that the function~$H^\delta(v^\delta)=\Zm_{\Od}(u^\delta,v^\delta)$ satisfies the massive harmonicity equation~\eqref{eq:mharm-def} everywhere in~$\Od$ except at the vertex~$u^\delta$ and that the mismatch in~\eqref{eq:mharm-def} equals to~$1$ if~$v^\delta=u^\delta$. The discrete integration by parts gives the identity}
\begin{align*}
\phi(u^\delta)\ &=\ {\delta^2\sum\nolimits_{v^\delta\in\Int \Od}\phi(v^\delta)\cdot \bigl(m^2H^\delta(v^\delta)-(1\!-\!m^2\delta^2)[\Delta^\delta H^\delta](v^\delta)\bigr)}\\
&=\ \delta^2 \sum\nolimits_{v^\delta\in\Int\Od}
{H^\delta(v^\delta)}\cdot\bigl(m^2\phi(v^\delta)-(1\!-\!m^2\delta^2)[\Delta^\delta\phi](v^\delta)\bigr).
\end{align*}
{We now pass to the limit $\delta\to 0$ in this identity; note that the prefactor $\delta^2$ in the right-hand side is nothing but the area of a unit cell on the grid $\delta\mathbb{Z}^2$.} Clearly,~$[\Delta^\delta\phi](v^\delta)=[\Delta\phi](v^\delta)+O(\delta\cdot\max_{v\in \Omega}|D^3\phi(v)|)$. The upper bound~\eqref{x:Zm-upper-bound} implies that the sums over~$\rho$-vicinities of~$u$ are uniformly (in~$\delta$) small as~$\rho\to 0$. Hence, the convergence of~{$H^\delta$ to~$h$} away from~$u$ implies that
\[
\phi(u)\ =\ \int_\Omega {h(v)}\bigl(m^2\phi(v)-[\Delta\phi](v)\bigr)dA(v).
\]

Therefore, each subsequential limit {of the functions $H^\delta$ coincides with~$G_\Omega^{(m)}(u,\cdot)$,} which proves \eqref{eq:Gm-conv} for fixed~$u$ and~$v$. The fact that the convergence is uniform follows from the equicontinuity of functions~$\Zm_{\Od}(u^\delta,v^\delta)$ discussed above and the compactness of the set of pairs~$(u,v)$ under consideration.
\end{proof}

\begin{remark} \label{rem:Gm/G>c}
It follows from the convergence~\eqref{eq:Gm-conv} that, for~$u,v\in\Omega\subset B(0,R)$, one has
\[
\exp(-c_0m^2R^2)\cdot G_\Omega(u,v)\ \le\ G^{(m)}_\Omega(u,v)\ \le\ G_\Omega(u,v) 
\]
due to the similar uniform estimate in discrete provided by Corollary~\ref{cor:Zm/Z>c}.
\end{remark}

\subsection{Convergence of martingale observables}\label{sect:m-convergence}
Recall that $(\Omega^{\delta};a^{\delta},b)$ are discrete approximations on scale $\delta$ of $(\Omega;a,b)$ in the Carath\'eodory sense. It follows from the absolute continuity of massive LERW with respect to the massless one (see~Section~\ref{sect:abs-cont}) that the family of mLERW probability measures in $(\Omega^{\delta};a^{\delta},b)$ is tight, when parameterized by the half-plane capacities of their conformal images (under the mappings~$\phi_{\contOd}$) in $(\mathbb{H};0,\infty)$. Using the Skorokhod representation theorem as in Section~\ref{sect:conv_m=0}, we can always assume that, almost surely,
\[
(\contOd_{t,r};a^\delta_{t,r})\ \mathop{\longrightarrow}\limits^{\mathrm{Cara}}\ (\Omega_{t,r};a_{t,r})\quad \text{as}\ \ \delta\to 0,
\]
where~$\Od_{t,r}=\Od\smallsetminus\gamma^\delta[0,n^\delta_{t,r}]$ and~$a^\delta_{t,r}=\gamma^\delta(n^\delta_{t,r})$. The goal of this section is to show that in this situation the martingale observables~\eqref{eq:defMm}, evaluated in the~$\frac{1}{2}r$-vicinity of~$b$, also converge almost surely to their continuous analogues. In other words, Proposition~\ref{prop:M=0conv} (for~$m=0$) and Proposition~\ref{prop:Mconv} (for~$m\ne 0$) are \emph{deterministic} statements, which we later apply for all possible limiting curves. For shortness, below we drop the second subscript $r$ and simply say that $t\le\tau_r$ instead.

We start by proving the convergence result for the classical (i.e., massless) LERW observable normalized at the boundary point~$b$.
\begin{proposition}\label{prop:M=0conv} In the setup described above, let~$t\le \tau_r$ and~$v\in B_\Omega(b,\frac{1}{2}r)$. Then
\[
M_{\Od_t}(v^\delta)\ =\ \frac{\Z_{\Od_t}(a_t^\delta,v^\delta)}{\Z_{\Od_t}(a_t^\delta,b)}\cdot \Z_{\Od}(o^\delta,b)
\ \to\ P_{\Omega_t}(a_t,v)\quad \text{as}\quad \delta\to 0,
\]
where the Poisson kernel~$P_{\Omega_t}(a_t,\cdot)$ in the domain~$\Omega_t$ is normalized so that one has~$P_{\Omega_t}(a_t,z)\sim P_{\Omega}(a,z)\sim G_\Omega(0,z)$ as~$z\to b$, see~\eqref{eq:PsimG} and Section~\ref{sect:estimates}.
\end{proposition}
\begin{proof} Given a small~$\varepsilon>0$, pick a point~$b_{\varepsilon r}\in B_\Omega(b,\varepsilon r)$. Corollary~\ref{cor:near-bdry=1} implies that
\[
\frac{\Z_{\Od_t}(a_t^\delta,v^\delta)\Z_{\Od}(o^\delta,b)}{\Z_{\Od_t}(a_t^\delta,b)}\ =\ \frac{\Z_{\Od_t}(a_t^\delta,v^\delta)\Z_{\Od}(o^\delta,b^\delta_{\varepsilon r})}{\Z_{\Od_t}(a_t^\delta,b^\delta_{\varepsilon r})}\cdot (1+O(\varepsilon^\beta)),
\]
with a universal exponent~$\beta>0$ and a universal (in particular, uniform in~$\delta$) \mbox{$O$-bound}. For each~$\varepsilon>0$, it follows from Corollary~\ref{cor:P-conv} and Corollary~\ref{cor:G-conv} that
\[
\frac{\Z_{\Od_t}(a_t^\delta,v^\delta)}{\Z_{\Od_t}(a_t^\delta,b^\delta_{\varepsilon r})}\ \mathop{\to}\limits_{\delta\to 0}\ \frac{P_{\Omega_t}(a_t,v)}{P_{\Omega_t}(a_t,b_{\varepsilon r})} \quad\text{and}\quad \Z_{\Od}(o^\delta,b^\delta_{\varepsilon r})\ \mathop{\to}\limits_{\delta\to 0}\ G_{\Omega}(0,b_{\varepsilon r}).
\]
Since we also know that
\[
\frac{P_{\Omega_t}(a_t,v)G_{\Omega}(0,b_{\varepsilon r})}{P_{\Omega_t}(a_t,b_{\varepsilon r})}\ \mathop{\to}\limits_{\varepsilon\to 0} \ P_{\Omega_t}(a_t,v)\frac{G_{\Omega}(0,b)}{P_{\Omega_t}(a_t,b)}\ =\ P_{\Omega_t}(a_t,v),
\]
the claim follows by first sending~$\delta\to 0$ and then~$\varepsilon\to 0$.
\end{proof}

We now move on to the convergence of the martingale observable in the massive setup. In order to formulate an analogue of Proposition~\ref{prop:M=0conv} in this situation, we need to introduce the massive Poisson kernel
\begin{equation}
\label{eq:defPm}
P^{(m)}_{\Omega_t}(a_t,z)\ :=\ P_{\Omega_t}(a_t,z)-m^2\int_{\Omega_t} P_{\Omega_t}(a_t,w)G^{(m)}_{\Omega_t}(w,z)dA(w).
\end{equation}
We refer the reader to Section~\ref{sect:estimates} (more precisely, to Remark~\ref{rem:PmQm-welldef}(i)), where the convergence of this integral is discussed; note that no regularity assumptions on~$\Omega_t$ are required for this fact.

\begin{proposition} \label{prop:Pm/Pconv} In the setup described above, let~$z\in \Omega_t$ (note that we do not need to assume that this point is close to~$b$). Then, as~$\delta\to 0$,  one has
\[
\frac{\Zm_{\Od_t}(a^\delta_t,z^\delta)}{\Z_{\Od_t}(a^\delta_t,z^\delta)}\ \to\ \frac{P_{\Omega_t}^{(m)}(a_t,z)}{P_{\Omega_t}(a_t,z)}\ =\ 1-m^2\int_{\Omega_t}\frac{P_{\Omega_t}(a_t,w)}{P_{\Omega_t}(a_t,z)}G_{\Omega_t}^{(m)}(w,z)dA(w).
\]
\end{proposition}
\begin{proof} {It follows from Lemma~\ref{lem:Zm(w,z)=Z-ZmZ} that}
\begin{equation}
\label{x:sum}
1-\frac{(1\!-\!m^2\delta^2)\Zm_{\Od_t}(a^\delta_t,z^\delta)}{\Z_{\Od_t}(a^\delta_t,z^\delta)}\ =\ m^2\delta^2\sum_{w^\delta\in\Int\Od_t} \frac{\Z_{\Od_t}(a^\delta_t,w^\delta)}{\Z_{\Od_t}(a^\delta_t,z^\delta)}\Zm_{\Od_t}(w^\delta,z^\delta).
\end{equation}
We now want to pass to the limit as~$\delta\to 0$ in this expression. For this purpose, we fix small parameters~$\rho,\rho_a>0$ and split the sum into the following four parts~$\mathrm{I}^\delta$--$\mathrm{IV}^\delta$; see Fig.~\ref{fig:sum} for an illustration:

\begin{figure}
\includegraphics[width=0.8\textwidth]{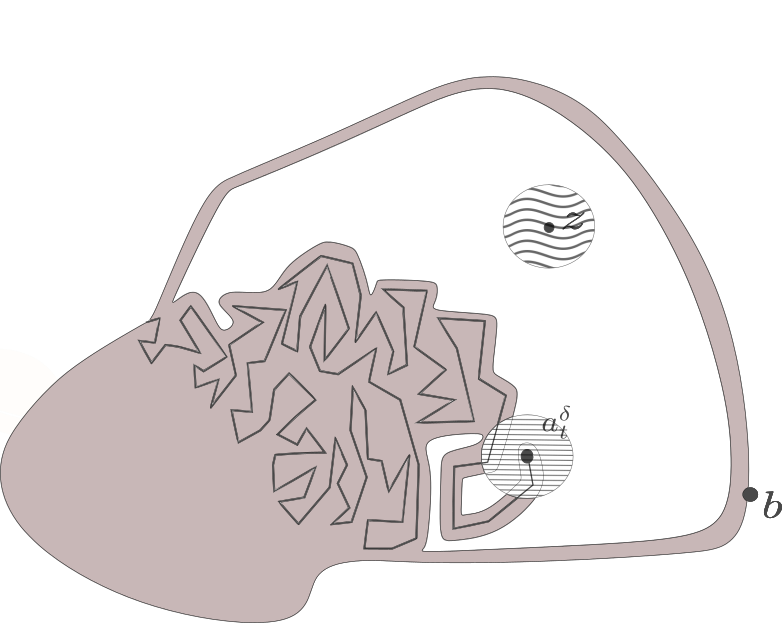}
	\caption{Four parts in the summation~\eqref{x:sum} over $w^{\delta}\in\Int\Omega_t^{\delta}$: the~white region inside the domain is~$\mathrm{I}^{\delta}$; the shaded vicinities of~$z$ and~$a^\delta_t$ are~$\mathrm{II}^{\delta}$ and~$\mathrm{III}^{\delta}$, respectively; the gray region is $\mathrm{IV}^{\delta}$.\label{fig:sum}}
\end{figure}

\smallskip

$\mathrm{I}^\delta$: \emph{sum over~$w^\delta$ lying jointly $\rho$-inside~$\Omega^\delta_t$ with~$z$ but not in~$B(z,\rho)\cup B(a_t^{\delta},\rho_a)$.} First, note that for~$w^\delta\not\in B(z,\rho)$ the summands are uniformly bounded from above due to Lemma~\ref{lem:ZZ/Z<cG} {since $Z^{(m)}_{\Omega^\delta_t}(w^\delta,z^\delta)\le Z_{\Omega^\delta_t}(w^\delta,z^\delta)$ and the right-hand side of~\eqref{eq:ZZ/Z<cG} is $O(1)$ provided that $v^\delta=a^\delta_t$ and $|w^\delta-z^\delta|\ge\rho$.} Thus, on these parts of~$\Od$ one can use Corollary~\ref{cor:P-conv} and~Proposition~\ref{prop:Gm-conv} to deduce the convergence
\begin{equation}
\label{x:sum1}
\mathrm{I}^\delta\ \to\ m^2\int_{\Omega_t^{(\rho)}\smallsetminus (B(z,\rho)\cup B(a_t,\rho_a))}\frac{P_{\Omega_t}(a_t,w)}{P_{\Omega_t}(a_t,z)}G^{(m)}_{\Omega_t}(w,z)dA(w)\quad\text{as}\ \ \delta\to 0,
\end{equation}
where~$\Omega_t^{(\rho)}$ denotes the connected component of the~$\rho$-interior of~$\Omega_t$ that contains~$z$.

\smallskip

$\mathrm{II}^\delta$: \emph{sum over {$w^\delta$ in} the $\rho$-vicinity of~$z$.} Due to the Harnack principle for discrete harmonic functions, the ratios~$\Z_{\Od_t}(a^\delta_t,w^\delta)/\Z_{\Od_t}(a^\delta_t,z^\delta)$ are uniformly bounded {if $w^\delta$ is close to}~$z^\delta$. Therefore, the summands of this part of~\eqref{x:sum} are majorated by the Green function~$\Z_{\Od_t}(w^\delta,z^\delta)$ since $\Z^{(m)}_{\Od_t}(w^\delta,z^\delta)\le\Z_{\Od_t}(w^\delta,z^\delta)$. Standard estimates
give
\begin{equation}
\label{x:sum2}
\mathrm{II}^\delta\ =\ O(\rho^2\,{\log \rho})\quad \text{uniformly~in}~\delta.
\end{equation}

$\mathrm{III}^\delta$: \emph{sum over the~$\rho_a$-vicinity of~$a_t$.} As already mentioned above,  Lemma~\ref{lem:ZZ/Z<cG} implies that on these parts of~$\Od$ the summands are uniformly bounded. Therefore,
\begin{equation}
\label{x:sum3}
\mathrm{III}^\delta\ =\ O(\rho_a^2)\quad \text{uniformly~in}~\delta.
\end{equation}

\smallskip

$\mathrm{IV}^\delta$: \emph{sum over~$w^\delta$ that are neither jointly $\rho$-inside~$\Omega^\delta_t$ with~$z$, nor in the $\rho_a$-vicinity~$a_t$, {nor in the} $\rho$-vicinity of $z$.} It is worth noting that these parts of~$\Omega^\delta_t$ can be in principle rather big as we require only the Carath\'eodory convergence of~$\Omega^\delta$ to~$\Omega$ (and so~$\Omega^\delta$ might contain big fjords that disappear in the limit). Nevertheless, one can easily see that the summands in this part of~\eqref{x:sum} are uniformly (in~$\delta$) small as~$\rho\to 0$. Indeed, due to~\eqref{eq:upper-bound-O3r} we have a uniform (provided that~$\delta$ is small enough) upper bound
\[
\frac{\Z_{\Od_t}(a^\delta_t,w^\delta)}{\Z_{\Od_t}(a^\delta_t,z^\delta)}\ \le\ C(\rho_a;\Omega_t,a_t)\quad \text{for}\ \ w^\delta\notin B_{\Od_t}(a_t^\delta,\rho_a).
\]
At the same time, since~$w^\delta$ is not $\rho$-jointly inside~$\Od_t$ with~$z$, there exists a ball of radius~$\rho$ which intersects the boundary of~$\Od_t$ and separates these two points in~$\Od_t$. Therefore, the weak-Beurling estimate (see Lemma~\ref{lem:weak Beurling}) gives that
\[
Z^{(m)}_{\Od_t}(w^\delta,z^\delta)=O(\rho^\beta),
\]
which allows us to conclude that
\begin{equation}
\label{x:sum4}
\mathrm{IV}^\delta\ \le\ \mathrm{Area(\Od_t)}\cdot C(\rho_a;\Omega_t,a_t)\cdot O(\rho^\beta)\quad\text{uniformly~in}~\delta.
\end{equation}
Combining~\eqref{x:sum1}--\eqref{x:sum4} together and sending first~$\rho\to 0$ and then~$\rho_a\to 0$ we get
\[
\mathrm{I}^\delta+\mathrm{II}^\delta+\mathrm{III}^\delta+\mathrm{IV}^\delta\ \to\ m^2\int_{\Omega_t}\frac{P_{\Omega_t}(a_t,w)}{P_{\Omega_t}(a_t,z)}G^{(m)}_{\Omega_t}(w,z)dA(w)\quad\text{as}\ \ \delta\to 0
\]
since the domains~$\Omega_t^{(\rho)}\smallsetminus (B(z,\rho)\cup B(a_t,\rho_a))$ exhaust~$\Omega_t$. The proof is completed.
\end{proof}

We now introduce the quantity
\begin{equation}
\label{eq:defNm}
N^{(m)}_{\Omega_t} = N^{(m)}_{\Omega_t}(a_t,b)\ :=\ \biggl[\,\frac{P^{(m)}_{\Omega_t}(a_t,b)}{P_{\Omega_t}(a_t,b)}\,\biggr]^{-1} =\ \biggl[\,\lim_{z\to b}\frac{P^{(m)}_{\Omega_t}(a_t,z)}{P_{\Omega_t}(a_t,z)}\,\biggr]^{-1},
\end{equation}
which keeps track of the normalization of the massive observable at the point~$b$. The existence of this limit {(as $z\to b$)} is discussed in Section~\ref{sect:estimates}; see~\eqref{eq:Pm/P(b)def}. The next proposition is the main result of this section. It is worth noting that { the convergence~\eqref{eq:Zm/Z(b)toPm/P(b)} discussed below and Corollary~\ref{cor:Zm/Z>c}, in particular, imply the uniform estimates~$1\le N^{(m)}_{\Omega_t}(a_t,b)\le\exp(c_0m^2R^2)$.}

\begin{proposition} \label{prop:Mconv}
In the setup of Proposition~\ref{prop:M=0conv} (i.e.,~$t\le\tau_r$ and~$v\in B_\Omega(b,\frac{1}{2}r)$), the following convergence holds true as~$\delta\to 0$:
\[
M^{(m)}_{\Od_t}(v^\delta)\ =\ \frac{\Zm_{\Od_t}(a_t^\delta,v^\delta)}{\Zm_{\Od_t}(a_t^\delta,b)}\cdot \Z_{\Od}(o^\delta,b)
\ \to\ P^{(m)}_{\Omega_t}(a_t,v)\cdot N^{(m)}_{\Omega_t}(a_t,b) =: M^{(m)}_{\Omega_t}(v),
\]
where the quantities in the right-hand side are defined by~\eqref{eq:defPm} and~\eqref{eq:defNm}.
\end{proposition}

\begin{proof} We start by generalizing the result of Proposition~\ref{prop:Pm/Pconv} to $z=b$:
\begin{equation}
\label{eq:Zm/Z(b)toPm/P(b)}
\frac{\Zm_{\Od_t}(a^\delta_t,b)}{\Z_{\Od_t}(a^\delta_t,b)}\ \to\ \frac{P_{\Omega_t}^{(m)}(a_t,b)}{P_{\Omega_t}(a_t,b)}\ =\
(N^{(m)}_{\Omega_t}(a_t,b))^{-1}\quad \text{as}\ \ \delta\to 0.
\end{equation}
We use the same argument as in the proof of Proposition~\ref{prop:M=0conv}. Given~$\varepsilon>0$ we pick a point~$b_{\varepsilon r}\in B_\Omega(b,\varepsilon r)$ and note that due to {Lemma \ref{lem:Zm(w,z)=Z-ZmZ}} and Corollary~\ref{cor:near-bdry=1} one has
\begin{align*}
1-\frac{(1\!-\!m^2\delta^2)\Zm_{\Od_t}(a^\delta_t,b)}{\Z_{\Od_t}(a^\delta_t,b)}\ &=\ m^2\delta^2\sum_{w^\delta\in\Int\Od_t} \frac{\Z^{(m)}_{\Od_t}(a^\delta_t,w^\delta)\Z_{\Od_t}(w^\delta,b)}{\Z_{\Od_t}(a^\delta_t,b)} \\
&=\ m^2\delta^2\sum_{w^\delta\in\Int\Od_t} \frac{\Z^{(m)}_{\Od_t}(a^\delta_t,w^\delta)\Z_{\Od_t}(w^\delta,b^\delta_{\varepsilon r})}{\Z_{\Od_t}(a^\delta_t,b^\delta_{\varepsilon r})}\cdot (1+O(\varepsilon^\beta))\\
&=\biggl[ 1-\frac{(1\!-\!m^2\delta^2)\Zm_{\Od_t}(a^\delta_t,b^\delta_{\varepsilon r})}{\Z_{\Od_t}(a^\delta_t,b^\delta_{\varepsilon r})}\biggr]\cdot (1+O(\varepsilon^\beta)),
\end{align*}
with a universal (and, in particular, uniform in~$\delta$) error term~$O(\varepsilon^\beta)$. Since~$\varepsilon>0$ can be chosen arbitrary small, Proposition~\ref{prop:Pm/Pconv} applied to~$z=b_{\varepsilon}$ implies~\eqref{eq:Zm/Z(b)toPm/P(b)}.

It remains to note that
\begin{align*}
M^{(m)}_{\Od_t}(v^\delta)\ & =\ \frac{\Zm_{\Od_t}(a_t^\delta,v^\delta)}{\Z_{\Od_t}(a_t^\delta,v^\delta)}\cdot \biggl[\frac{\Zm_{\Od_t}(a^\delta_t,b)}{\Z_{\Od_t}(a^\delta_t,b)}\biggr]^{-1}\!\cdot \frac{\Z_{\Od_t}(a_t^\delta,v^\delta)\Z_{\Od}(o^\delta,b)}{\Z_{\Od_t}(a^\delta_t,b)} \\
&\to\ \frac{P_{\Omega_t}^{(m)}(a_t,v)}{P_{\Omega_t}(a_t,v)}\cdot N^{(m)}_{\Omega_t}(a_t,b)\cdot P_{\Omega_t}(a_t,v)\quad \text{as}\ \ \delta\to 0
\end{align*}
due to~Proposition~\ref{prop:Pm/Pconv}, convergence \eqref{eq:Zm/Z(b)toPm/P(b)}, and Proposition~\ref{prop:M=0conv}, respectively.
\end{proof}

\section{Estimates and computations in continuum} \label{sect:continuum}

\setcounter{equation}{0}

For shortness, from now onwards we drop a boundary point $a$ from the notation of Poisson kernels since there is only one point~$a_t$ (tip of the slit) that we are interested in when speaking about domains $\Omega_t=\Omega\smallsetminus\gamma[0,t]$.

\subsection{A priori estimates and massive Poisson kernels} \label{sect:estimates} Given a simply connected domain~$\Lambda\subset B(0,R)$ and its uniformization~$\phi_\Lambda:\Lambda\to\mathbb H$, we set
\begin{equation}
\label{eq:PQdef}
P_\Lambda(z):=-\frac{1}{\pi}\Im\frac{1}{\phi_\Lambda(z)}\,,\qquad Q_\Lambda(z):=-\frac{1}{\pi}\Im\frac{1}{(\phi_\Lambda(z))^2}\,.
\end{equation}
It is worth emphasizing that this definition heavily relies upon the choice of~$\phi_\Lambda$ (namely, on the choice of~$a=\phi^{-1}_\Lambda(0)$ and the normalization of~$\phi_\Lambda$ at~$b=\phi_\Lambda^{-1}(\infty)$), which is not mentioned explicitly in the notation. In particular, one has
\begin{equation}
\label{eq:Q(b)/P(b)=0}
\frac{Q_\Lambda(b)}{P_\Lambda(b)}\ =\ \lim_{z\to b}\frac{Q_\Lambda(z)}{P_\Lambda(z)}\ =\ 0
\end{equation}
(see Section~\ref{sect:boundarybehavior} for the discussion of the existence of the limit). Recall that by~$G_\Lambda(w,z)$ we denote the \emph{positive} Green function in~$\Lambda$ and that~$G_\Lambda^{(m)}(w,z)$ stands for the massive Green function discussed in Section~\ref{sect:mGreen}, i.e. the integral kernel of the operator~$(-\Delta_\Lambda+m^2)^{-1}$, where~$\Delta_\Lambda$ denotes the Laplacian in~$\Lambda$ with Dirichlet boundary conditions. As mentioned in Remark~\ref{rem:Gm/G>c}, for all~$w,z\in \Lambda$ one has
\begin{equation}
\label{eq:Gbound}
\exp(-c_0m^2R^2)\cdot G_\Lambda(w,z)\ \le\ G^{(m)}_\Lambda(w,z)\ \le\ G_\Lambda(w,z)\ \le\ \frac{1}{2\pi}\log\frac{2R}{|w-z|}\,.
\end{equation}
Since~$-\Delta G^{(m)}_\Lambda(w,\cdot)=\delta_w(\cdot)-m^2G^{(m)}_\Lambda(w,\cdot)$, one has the identity
\begin{equation}
\label{eq:Gm=G-GGm}
G^{(m)}_\Lambda(w,z)\ =\ G_\Lambda(w,z)-m^2\int_\Omega G_\Lambda(w,w')G^{(m)}_\Lambda(w',z)dA(w').
\end{equation}
Note that the identity~\eqref{eq:Gm=G-GGm} is nothing but a continuous counterpart of the similar identity~\eqref{eq:Zm(w,z)=Z-ZmZ} for the partition functions of random walks discussed in Section~\ref{sect:notation}.

\begin{lemma} There exists an absolute constant~$C>0$ such that, for each simply connected domain~$\Lambda\subset \mathbb C$, its uniformization~$\phi_\Lambda:\Lambda\to\mathbb H$, and~$z,w\in\Lambda$, the following estimates are fulfilled:
\begin{equation}
\label{eq:PQGbound}
\biggl|\frac{P_\Lambda(w)}{P_\Lambda(z)}-1\biggr|\cdot G_\Lambda(w,z)\le C,\qquad \biggl|\frac{Q_\Lambda(w)}{P_\Lambda(w)}-\frac{Q_\Lambda(z)}{P_\Lambda(z)}\biggr|\cdot \frac{G_\Lambda(w,z)}{P_\Lambda(z)}\le C.
\end{equation}
\end{lemma}
\begin{proof}
It is easy to see that both expressions are invariant under M\"obius automorphisms of~$\mathbb H$ preserving the point~$0$. Therefore, one can assume~$\phi_\Lambda(z)=i$ without loss of generality. In this situation, the required estimates~\eqref{eq:PQGbound} are nothing but the claim that both functions
\begin{align*}
&|\Im \phi_\Lambda(w)^{-1}+1|\cdot G_{\mathbb H}(\phi_\Lambda(w),i)\qquad \text{and} \\
& \frac{|\Im \phi_\Lambda(w)^{-2}|}{|\Im \phi_\Lambda(w)^{-1}|}\cdot G_{\mathbb H}(\phi_\Lambda(w),i)= \frac{2|\Re \phi_\Lambda(w)|}{|\phi_\Lambda(w)|^3}\cdot G_{\mathbb H}(\phi_\Lambda(w),i),
\end{align*}
are bounded in the upper half-plane, which is clearly true since both of them are continuous in~$\phi=\phi_\Lambda(w)\in\overline{\mathbb H}$ (including at the point~$i$) and decay as~$|\phi|\to\infty$.
\end{proof}
\begin{remark} For later purposes, it is useful to rewrite~\eqref{eq:PQGbound} as
\begin{align}
P_\Lambda(w)G_\Lambda(w,z)&\le P_\Lambda(z)G_\Lambda(w,z)+CP_\Lambda(z), \label{eq:PGbound}\\
|Q_\Lambda(w)|G_\Lambda(w,z)&\le CP_\Lambda(z)P_\Lambda(w)+P_\Lambda(w)G_\Lambda(w,z)(P_\Lambda(z))^{-1}|Q_\Lambda(z)| \notag \\
&\le CP_\Lambda(z)P_\Lambda(w)+|Q_\Lambda(z)|G_\Lambda(w,z)+C|Q_\Lambda(z)|. \label{eq:QGbound}
\end{align}
\end{remark}
We now introduce massive counterparts of the functions~\eqref{eq:PQdef} as follows:
\begin{align}
\label{eq:Pmdef} P_\Lambda^{(m)}(z)&:=P_\Lambda(z)-m^2\int_{\Lambda}P_\Lambda(w)G_\Lambda^{(m)}(w,z)dA(w),\\
\label{eq:Qmdef} Q_\Lambda^{(m)}(z)&:=Q_\Lambda(z)-m^2\int_{\Lambda}Q_\Lambda(w)G_\Lambda^{(m)}(w,z)dA(w).
\end{align}
\begin{remark}
\label{rem:PmQm-welldef} (i) The estimate~\eqref{eq:PGbound} ensures that the massive Poisson kernel~$P_\Lambda^{(m)}(z)$ is well-defined since the only possible pathology in the integral is at~$w=z$, where the integrand is bounded from above by a multiple of the Green function~$G_\Lambda(w,z)$. Moreover, one easily sees that
\begin{equation}\label{eq:Pm/P>c}
\exp(-c_0m^2R^2){P_t(a_t,z)}\le{P_t^{(m)}(a_t,z)}\le{P_t(a_t,z)}
\end{equation}
due to Proposition~\ref{prop:Pm/Pconv} and similar uniform bounds provided by Corollary~\ref{cor:Zm/Z>c}.

\smallskip

\noindent (ii) On the contrary, \eqref{eq:QGbound} {only} guarantees that the function $Q_\Lambda^{(m)}$ is well-defined under the additional assumption~$\int_\Lambda P_\Lambda(w)dA(w)<+\infty$. Though this is not always true in general, it follows from Corollary~\ref{cor:intPconv}(i) given below that this assumption holds for almost all (in~$t$) domains~$\Lambda=\Omega_t$ generated by a Loewner evolution in~$\Omega$.
\end{remark}

\begin{lemma} The following identity is fulfilled for all~$z\in\Lambda$:
\begin{equation}
\label{eq:Pmdef1}
P_\Lambda^{(m)}(z)=P_\Lambda(z)-m^2\int_{\Lambda}P_\Lambda^{(m)}(w)G_\Lambda(w,z)dA(w).
\end{equation}
\end{lemma}

\begin{proof} Note that the integral converges due to~\eqref{eq:PGbound} and since~$P_\Lambda^{(m)}(w)\le P_\Lambda(w)$. Moreover, one has
\begin{align*}
& \int_{\Lambda}P_\Lambda^{(m)}(w)G_\Lambda(w,z)dA(w)\\ & =\
\int_{\Lambda}\biggl[P_\Lambda(w)-m^2\int_{\Lambda}
P_\Lambda(w')G_\Lambda^{(m)}(w',w)dA(w')\biggr]G_\Lambda(w,z)dA(w)\\
& =\ \int_\Lambda P_\Lambda(w)\biggl[G_\Lambda(w,z)-\int_\Lambda G_\Lambda^{(m)}(w,w')G_\Lambda(w',z)dA(w')\biggr]dA(w)\\
& =\ \int_\Lambda P_\Lambda(w)G_\Lambda^{(m)}(w,z)dA(w)\ =\ P_\Lambda^{(m)}(z),
\end{align*}
where the application of the Fubini theorem in the second equality is based upon the uniform estimate
\[
P_\Lambda(w)G_\Lambda^{(m)}(w,w')G_\Lambda(w',z)\le P_\Lambda(z)(G_\Lambda(w,w')+C)(G_\Lambda(w',z)+C)
\]
which follows from~\eqref{eq:PGbound}.
\end{proof}
Assume now that~$b:=\phi_\Lambda^{-1}(\infty)$ is a degenerate prime end of~$\Lambda$. The representation~\eqref{eq:Pmdef1} together with the discussion given in Section~\ref{sect:boundarybehavior} allows one to define the following quantity (note that here and below we abuse the notation in a way similar to Section~\ref{sect:boundarybehavior} when writing the ratio of two functions, both satisfying Dirichlet boundary conditions, at a boundary point~$b$):
\begin{equation}
\label{eq:Pm/P(b)def}
\frac{P_\Lambda^{(m)}(b)}{P_\Lambda(b)}\ :=\ \lim_{z\to b} \frac{P_\Lambda^{(m)}(z)}{P_\Lambda(z)} \ =\ 1-m^2\int_\Lambda P_\Lambda^{(m)}(w)\frac{G_\Lambda(w,b)}{P_\Lambda(b)}dA(w).
\end{equation}
Indeed, one can exchange the limit~$z\to b$ and the integration over~$w\in\Lambda$ due to the uniform estimate~\eqref{eq:PGbound}, which provides a majorant
\[
P_\Lambda^{(m)}(w)\frac{G_\Lambda(w,z)}{P_\Lambda(z)}\ \le\ \frac{P_\Lambda(w)G_\Lambda(w,z)}{P_\Lambda(z)}\ \le\  G_\Lambda(w,z)+C,
\]
and the fact that~$\max_{z\in B_\Lambda(b,r)}\int_{B_{\Lambda}(b,2r)}G_\Lambda(w,z)dA(w)\to 0$ as~$r\to 0$, which follows from~\eqref{eq:Gbound} and allows one to neglect the contributions of vicinities of the point~$b$ (where the Green function blows up and thus no uniform in~$z$ majorant is available).

\subsection{Hadamard's formula}\label{sect:Hadamard} We now move to the Loewner equation setup and assume that a decreasing family of subdomains~$\Omega_t\subset\Omega$ is constructed according to~\eqref{eq:Loewner} and that their uniformizations onto the upper half-plane are fixed as
\[
\phi_t:=(g_t-\xi_t)\circ\phi_\Omega:\Omega_t\to\mathbb H
\]
so that, in particular,~$\phi_t(a_t)=0$ and~$\phi_t(b)=\infty$.  For shortness, from now onwards we replace the subscript~$\Omega_t$ by~$t$, thus we write $G_t(w,z)$ instead of $G_{\Omega_t}(w,z)$, $P_t(z)$~instead of $P_{\Omega_t}(z)=P_{\Omega_t}(a_t,z)$, etc. The following lemma is classical.
\begin{lemma}[Hadamard's formula] \label{lem:Hadamard}
For each~$z,w\in\Omega$ the function~$G_t(z,w)$ is differentiable in~$t$ (until the first moment when either~$z\not\in\Omega_t$ or~$w\not\in\Omega_t$) and
\begin{equation}\label{eq:Hadamard}
\partial_tG_t(w,z)=-2\pi P_t(w)P_t(z).
\end{equation}
\end{lemma}
\begin{proof} Let~$w_{\mathbb{H}}:=\phi_\Omega(w)$ and~$z_{\mathbb{H}}:=\phi_\Omega(z)$, note that one has
\[
G_t(w,z)=-\frac{1}{2\pi}\log\biggl|\frac{g_t(w_{\mathbb{H}})-g_t(z_{\mathbb{H}})}{g_t(w_{\mathbb{H}})-\overline{g_t(z_{\mathbb{H}})}}\biggr|.
\]
Since both~$g_t(w_{\mathbb{H}})$ and~$g_t(z_{\mathbb{H}})$ satisfy the Loewner equation~\eqref{eq:Loewner}, one easily obtains
\begin{align*}
\partial_t{G}_t(w,z)
&=-\frac{1}{2\pi}\Re\biggl[\frac{\partial_t{g}_t(w_{\mathbb{H}})-\partial_t{g}_t(z_{\mathbb{H}})}{g_t(w_{\mathbb{H}})-g_t(z_{\mathbb{H}})}-\frac{\partial_t{g}_t(w_{\mathbb{H}})- \overline{\partial_t{g}_t(z_{\mathbb{H}})}}{g_t(w_{\mathbb{H}})-\overline{g_t(z_{\mathbb{H}})}}\biggr]\\
&=\frac{1}{\pi}\Re\biggl[\frac{1}{(g_t(w_{\mathbb{H}})\!-\!\xi_t)(g_t(z_{\mathbb{H}})\!-\!\xi_t)}-\frac{1}{(g_t(w_{\mathbb{H}})\!-\!\xi_t)\overline{(g_t(z_{\mathbb{H}})\!-\!\xi_t)}}\biggr]\\
&=-\frac{2}{\pi}\Im\biggl[\frac{1}{g_t(w_{\mathbb{H}})-\xi_t}\biggr]\Im\biggl[\frac{1}{g_t(z_{\mathbb{H}})-\xi_t}\biggr]\ =\ -2\pi P_t(w)P_t(z).\qedhere
\end{align*}
\end{proof}

As pointed out in~\cite{Makarov-Smirnov}, it immediately follows from the Hadamard formula that the integrals~$\int_{\Omega_t}P_t(w)dA(w)$ converge for almost all~$t$, see the next corollary. In our analysis we also need a stronger estimate which guarantees the convergence of integrals~$\int_{\Omega_t}(P_t(w))^2dA(w)$ for almost all~$t$ provided that~$\gamma$ is an SLE(2) curve.

\begin{corollary}\label{cor:intPconv} (i) In the same setup, one has
\[
\int_0^\infty\biggl[\int_{\Omega_t}P_t(w)dA(w)\biggr]^{2}dt\ \le\ \frac{1}{2\pi}\int_\Omega\int_\Omega G_0(w,z)dA(w)dA(z)\ <\ +\infty.
\]
(ii) Moreover, if~$\gamma$ is {an SLE($\kappa$), $\kappa\le 4$, curve} running from~$a$ to~$b$ in~$\Omega$, then
\[
\int_0^\infty\int_{\Omega_t}(P_t(w))^2dA(w)dt\ <\ +\infty\quad \text{almost~surely}.
\]
\end{corollary}

\begin{proof} {(i) Given a point~$z\in\Omega$, let $\tau_z:=\inf\{t>0:z\not\in \Omega_t\}$ be the time when~$z$ is hit or swallowed by the curve~$\gamma$ (as usual, we set $\tau_z:=+\infty$ if this does not happen). By integrating the Hadamard formula~\eqref{eq:Hadamard} in~$t$ one easily sees that}
\[
2\pi\int_0^{{\tau_w\wedge\tau_z}}\!\!P_t(w)P_t(z)dt\ \le\ G_0(w,z).
\]
The claim follows by integrating {this inequality over~$w,z\in\Omega$ since}
\[
\int_{\Omega\times\Omega}\int_0^{\tau_w\wedge\tau_z}\!\!P_t(w)P_t(z)dtdA(w)dA(z)\ =\
\int_0^\infty\int_{\Omega_t\times\Omega_t}\!\!P_t(w)P_t(z)dA(w)dA(z)dt.
\]

\smallskip

\noindent (ii) Given a planar (simply connected) domain~$\Lambda$ and~$w\in\Lambda$, let
\[
G^*_\Lambda(w,w)\ :=\ \lim_{z\to w}\bigl(G_\Lambda(w,z)+\tfrac{1}{2\pi}\log|z\!-\!w|\bigr)\ =\ \tfrac{1}{2\pi}\log\crad_\Lambda(w),
\]
where~$\crad_\Omega(w)$ denotes the \emph{conformal radius} of the point~$w$ in~$\Omega$.
A straightforward generalization of Lemma~\ref{lem:Hadamard} implies that $\partial_tG^*_t(w,w)=-2\pi(P_t(w))^2$ for~$w\in\Omega_t$.

Since SLE($\kappa$) curves {with $\kappa\le 4$} are not self-touching, for all $w\in\Omega$ one almost surely has \mbox{$w\in\Omega_t=\Omega\smallsetminus\gamma[0,t]$} for all~$t\le\infty$. {Applying the Fubini theorem as above, we obtain the identity}
\begin{align*}
\int_0^\infty\int_{\Omega_t}(P_t(w))^2dA(w)dt\ &=\ 
\frac{1}{2\pi}\int_\Omega \log\frac{\crad_{\Omega}(w)}{\crad_{\Omega\smallsetminus\gamma[0,\infty]}(w)}\,dA(w),
\end{align*}
where we slightly abuse the notation in the denominator: $\crad_{\Omega\smallsetminus\gamma[0,\infty]}(w)$ stands for the conformal radius of~$w$ in one of the two components of~$\Omega\smallsetminus\gamma[0,\infty]$ to which this point belongs. Standard estimates (e.g., see~\cite[Section~5.3.6.2]{kemppainen-book}) for the SLE($\kappa$) curves $\phi_\Omega(\gamma)$ \emph{in the upper half-plane $\mathbb H$} imply
\[
\mathbb E\biggl[\,\log\frac{\crad_\Omega(w)}{\crad_{\Omega\smallsetminus\gamma[0,\infty]}(w)}\,\biggr]\ =\ \mathbb E\biggl[\,\log\frac{\crad_{\mathbb H}(\phi_\Omega(w))}{\crad_{\mathbb H\smallsetminus\phi_\Omega(\gamma[0,\infty])}(\phi_\Omega(w))}\,\biggr]\ \le\ \const,
\]
uniformly over $w\in\Omega$. Therefore,
\[
\mathbb E\biggl[\,\int_0^\infty\int_{\Omega_t}(P_t(w))^2dA(w)dt\,\biggr]\ \le\ \const\cdot\mathrm{Area}(\Omega)
\]
and, in particular, this integral is finite almost surely.
\end{proof}

We now derive a counterpart of Lemma~\ref{lem:Hadamard} in the massive setup.
\begin{lemma}[massive Hadamard's formula] \label{lem:mHadamard}
In the same setup, the massive Green function~$G_t^{(m)}(w,z)$ is differentiable in~$t$ (until the first moment when either~$z\not\in\Omega_t$ or~$w\not\in\Omega_t$) and
\begin{equation}\label{eq:mHadamard}
\partial_tG^{(m)}_t(w,z)=-2\pi P^{(m)}_t(w)P^{(m)}_t(z),
\end{equation}
where the massive Poisson kernels~$P^{(m)}_t(w)$,~$P^{(m)}(z)$ in~$\Omega_t$ are given by~\eqref{eq:Pmdef}.
\end{lemma}
\begin{proof} It is easy to see that the increments of~$G^{(m)}_t(w,z)$ are bounded by those of~$G_t(w,z)$: e.g., this follows from Proposition~\ref{eq:Gm-conv}, Corollary~\ref{cor:G-conv} and the similar inequality in discrete which is trivial. Therefore, $G^{(m)}_t(w,z)$ is an {absolutely} continuous function of~$t$, the derivative~$\partial_tG^{(m)}_t(w,z)$ exists (given $w,z$) for almost all~$t$ and
\begin{equation}
\label{eq:x-paGm<paG<PP}
0 \le -\partial_tG^{(m)}_t(w,z)\le -\partial_tG_t(w,z) \le 2\pi P_t(w)P_t(z).
\end{equation}
{Due to the Tonelli theorem, for almost all $t$ the derivative $\partial_t G^{(m)}_t(w,z)$ also exists simultaneously for almost all~$w,z\in\Omega$. Moreover, note that it is sufficient to prove~\eqref{eq:mHadamard} for \emph{almost} all $t,w,z$: if this is done, the same claim for all $t,w,z$ (and, in particular, the existence of $\partial_t G^{(m)}_t(w,z)$ for \emph{all} $t,w,z$) follows from the continuity of the massive Green function and massive Poisson kernels in $(t,w,z)$.}

\smallskip

Differentiating in~$t$ the resolvent identity (see~\eqref{eq:Gm=G-GGm})
\[
G^{(m)}_t(w,z)\ =\ G_t(w,z)-m^2\int_\Omega G_t(w,w')G^{(m)}_t(w',z)dA(w')
\]
(since {both}~$G_t(w,w')$ and~$G_t^{(m)}(w',z)$ are monotone in~$t$, this {differentiation can be justified by} the Tonelli theorem) and using Lemma~\ref{lem:Hadamard} one obtains
\begin{align*}
\partial_t G^{(m)}_t(w,z)\ &=\ -2\pi P_t(w)P_t(z)+2\pi m^2\int_\Omega P_t(w)P_t(w')G^{(m)}_t(w',z)dA(w')\\
&\phantom{\ -2\pi P_t(w)P_t(z)+2\pi}-m^2\int_\Omega G_t(w,w')\partial_t G^{(m)}_t(w',z)dA(w')\\
&=\ -2\pi P_t(w)P^{(m)}_t(z)-m^2\int_\Omega G_t(w,w')\partial_t G^{(m)}_t(w',z)dA(w').
\end{align*}
Denote by~$\mathfrak{G}_t=(-\Delta)^{-1}$ and~$\mathfrak{G}^{(m)}_t=(-\Delta+m^2)^{-1}$ integral operators acting on test function~$h:\Omega_t\to\mathbb R$ as follows:
\begin{align*}
(\mathfrak{G}_th)(w)\ &:=\ \int_{\Omega_t}h(w')G_t(w',w)dA(w'),\\
(\mathfrak{G}_t^{(m)}h)(w)\ &:=\ \int_{\Omega_t}h(w')G^{(m)}_t(w',w)dA(w').
\end{align*}
In this notation, we can rewrite the equation for the derivative~$\partial_t G^{(m)}_t(w,z)$ obtained above as
\[
(\mathrm{Id}+m^2\mathfrak{G}_t)(\partial_t G^{(m)}_t(\cdot,z))\ =\ -2\pi P_t(w)P^{(m)}_t(z).
\]
The resolvent identity (see~\eqref{eq:Gm=G-GGm}) {reads as $\mathfrak{G}^{(m)}_t\!=\mathfrak{G}_t-m^2\mathfrak{G}^{(m)}_t\mathfrak{G}_t$, provided that the integrals under consideration converge and that the Fubini theorem can be applied. Therefore,}
\[
(\mathrm{Id}-m^2\mathfrak{G}^{(m)}_t)(\mathrm{Id}+m^2\mathfrak{G}_t)h=h,\qquad h=\partial_t G^{(m)}_t(\cdot,z),
\]
{where the use of the Fubini theorem can be justified via the estimates~\eqref{eq:x-paGm<paG<PP} and~\eqref{eq:PGbound}.} Therefore, for almost all~$t$ (and, given~$t$, for almost all $w,z$), one has
\[
\partial_t G^{(m)}_t(w,z)\ =\ -2\pi\big[(\mathrm{Id}-m^2\mathfrak{G}^{(m)}_t)P_t\big](w) P^{(m)}_t(z)\ =\ -2\pi P^{(m)}_t(w)P^{(m)}_t(z),
\]
{which is nothing but the identity~\eqref{eq:mHadamard}. As already mentioned above, the similar claim for \emph{all}~$t,w,z$ follows from the continuity of the massive Green function and massive Poisson kernels in all the variables $t,w,z$.}
\end{proof}

\subsection{Driving term of mSLE(2)}\label{sect:computations}
With the above estimates and Hadamard's formula, we are now prepared to compute the driving term of massive SLE$(2)$ such that the normalized massive Poisson kernel is a martingale.
Recall that {under each probability measure $\mathbb{P}^{(m)}_{(\Omega;a,b)}$ (obtained as a subsequential weak limit of mLERWs)} we have
\[
\mathrm{d}\xi_t\ =\ \sqrt{2}(\mathrm{d}B_t+\mathrm{d}\langle B,L^{(m)}\rangle_t)\ {=\,\sqrt{2}\mathrm{d}B_t + 2\lambda_t\mathrm{d}t\,,}
\]
where $B_t$ is a standard Brownian motion and the process~$L^{(m)}_t$ comes from the Girsanov theorem as explained in Section~\ref{sect:abs-cont}. {Our goal is to identify the drift term $2\lambda_t\mathrm{d}t$; note that we will do this using the martingale property of the processes $t\mapsto M^{(m)}_t(z)$, $z\in\Omega$, and \emph{not} through identifying the process~$L^{(m)}_t$; cf. Remark~\ref{rem:Nm-not-mart} and Remark~\ref{rem:Nm-vs-Nb}. More precisely, following the strategy indicated in~\cite{Makarov-Smirnov} we
\begin{itemize}
\item analyze the random processes $t\mapsto P^{(m)}_t(z)$, $z\in B_\Omega(b,\tfrac{1}{2})$, relying upon the convolution formula~\eqref{eq:Pmdef}, the massive Hadarmard formula (Lemma~\ref{lem:mHadamard}), and a version of the stochastic Fubini theorem (see Lemma~\ref{lem:fubini-ok} below);
\item use the martingale property of the process $t\mapsto M^{(m)}_t(z)=P^{(m)}_t(z)N^{(m)}_t$ for each $z\in B_\Omega(b,\frac{1}{2}r)$ in order \begin{itemize}
    \item to analyze the process $t\mapsto N^{(m)}_t$ (this is done in Lemma~\ref{lem:dNmt=}(i)) and
    \item to identify the drift term $2\lambda_t\mathrm{d}t$ of the process $\xi_t$ (see Lemma~\ref{lem:dNmt=}(ii)).
    \end{itemize}
\end{itemize}

Recall that we use the notation $\mathrm{d}$ in the stochastic calculus/SDE context and the notation $dA$ for the Lebesque measure in $\Omega$, over which we often integrate in the following computations.} For each~$w\in\Omega$, the process~${t\mapsto}\,g_t(\phi_\Omega(w))$ satisfies the Loewner equation~\eqref{eq:Loewner}, thus one has
\begin{equation}\label{eq:dP=}
\mathrm{d}P_t(w)\ =\ 
-\frac{1}{\pi}\Im\biggl[\frac{\mathrm{d}\xi_t}{(g_t(\phi_\Omega(w))-\xi_t)^2}+ \frac{\mathrm{d}\langle\xi,\xi\rangle_t-2\mathrm{d}t}{(g_t(\phi_\Omega(w))-\xi_t)^3}\biggr]\ =\ Q_t(w)\mathrm{d}\xi_t\,.
\end{equation}

We want to substitute this expression (together with the massive Hadamard formula~\eqref{eq:mHadamard}) into the definition~\eqref{eq:Pmdef} of the massive Poisson kernel. The following lemma handles the question of interchanging the stochastic integration over the continuous semi-martingale~$\xi_t$ with the Lebesgue integration over~$w\in\Omega$.
\begin{lemma}\label{lem:fubini-ok} The process~$\int_\Omega Q_t(w)G_t^{(m)}(w,z)dA(w)$ is a local semi-martingale. Moreover, almost surely, for all~$T>0$ the following identity is fulfilled:
\[
\int_\Omega \biggl[\,\int_0^T Q_t(w)G_t^{(m)}(w,z)\mathrm{d}\xi_t\biggr] dA(w) \ =\ \int_0^T\biggl[\,\int_\Omega Q_t(w)G_t^{(m)}(w,z)dA(w)\biggr]\mathrm{d}\xi_t.
\]
\end{lemma}
\begin{proof} We use a version of the stochastic Fubini theorem given in~\cite{veraar}. In order to apply this result, one needs to check that the following two conditions hold almost surely {(recall that $\mathrm{d}\xi_t=\sqrt{2}(\mathrm{d}B_t+\mathrm{d}\langle B,L^{(m)}\rangle_t)$; in particular, $\mathrm{d}\langle\xi,\xi\rangle_t=2\mathrm{d}t$):}
\begin{align}
\label{x:mart-part}\int_\Omega \biggl[\int_0^T \big|Q_t(w)G_t^{(m)}(w,z)\big|^2{\mathrm{d}t}
& \biggr]^{1/2}\!dA(w)<+\infty,\\
\label{x:bv-part}\int_\Omega \biggl[\int_0^T \big|Q_t(w)G_t^{(m)}(w,z){\mathrm{d}\langle B,L^{(m)}\rangle_t\big|} & \biggr] dA(w)<+\infty.
\end{align}
The first estimate~\eqref{x:mart-part} can be easily derived from Corollary~\ref{cor:intPconv}(ii) (and {from} the absolute continuity of mSLE(2) with respect to SLE(2) discussed in Section~\ref{sect:abs-cont}) since the uniform bound~\eqref{eq:QGbound} implies
\begin{align*}
|Q_t(w)G_t^{(m)}(w,z)|^2\ &\le\ (CP_t(z)P_t(w)+|Q_t(z)|G_t(w,z)+C|Q_t(z)|)^2\\
& \le\ C(z)(P_t(w)^2+G_t(w,z)^2+1),
\end{align*}
where~$C(z):= 3C^2\max_{t\in [0,T]}\{(P_t(z))^2+|Q_t(z)|^2\}<+\infty$ almost surely. In its turn, the second estimate~\eqref{x:bv-part} follows from~\eqref{x:mart-part} and the Kunita--Watanabe inequality {(see \cite[Proposition~4.5]{legall-book})} as~$\langle L^{(m)},L^{(m)}\rangle_T<+\infty$ almost surely; see~\eqref{eq:LL<infty}.
\end{proof}
Using~\eqref{eq:dP=}, the massive Hadamard formula~\eqref{eq:mHadamard} and Lemma~\ref{lem:fubini-ok}, we conclude that, for each~$z\in\Omega$, the random process~$P^{(m)}(z)$ is a local semi-martingale and
\begin{align}
\notag & \mathrm{d}P^{(m)}_t(z)\ =\ \mathrm{d}P_t(z)-m^2\int_{\Omega_t}\left(G^{(m)}_t(w,z)\mathrm{d}P_t(w)+P_t(w)\mathrm{d} G^{(m)}_t(w,z)\right)dA(w)\\
\notag &=\ Q_t(z)\mathrm{d}\xi_t-m^2\int_{\Omega_t}\left(Q_t(w)G^{(m)}_t(w,z)\mathrm{d}\xi_t-2\pi P_t(w)P^{(m)}_t(w)P^{(m)}_t(z)\mathrm{d}t\right)dA(w)\\
&=\ Q^{(m)}_t(z)\mathrm{d}\xi_t + 2\pi m^2 P^{(m)}_t(z)\left[\int_{\Omega_t}P_t(w)P^{(m)}_t(w)dA(w)\right]\mathrm{d}t.
\label{eq:dPm=}
\end{align}

\smallskip

{We now move to the key part of the computation. Recall that
\[
N^{(m)}_t\ =\ P_t(b)/P^{(m)}_t(b)\ =\ M^{(m)}_t(z)/P^{(m)}_t(z),\quad z\in\Omega_t
\]
and note that the process $N^{(m)}_t$ is a semi-martingale since $M^{(m)}_t(z)$ is a martingale and $P^{(m)}_t(z)$ is a (strictly positive) semi-martingale.}

\begin{lemma}\label{lem:dNmt=} {(i) The positive semi-martingale $N^{(m)}_t$ satisfies the following SDE:
\begin{equation}
\label{eq:dN=}
\mathrm{d}N^{(m)}_t\ =\ -N^{(m)}_t\biggl[\frac{Q^{(m)}_t(b)}{P^{(m)}_t(b)}\sqrt{2}\mathrm{d}B_t + 2\pi m^2\left[\int_{\Omega_t}P_t(w)P^{(m)}_t(w)dA(w)\right]\mathrm{d}t\biggr].
\end{equation}
\noindent (ii) The following identity for the drift term of the driving process $\xi_t$ holds:}
\begin{equation}
\label{eq:dlambda=}
2\lambda_t\mathrm{d}t\ =\ -
\frac{\mathrm{d}\langle\xi,N^{(m)}\rangle_t}{N_t^{(m)}}\ =\ 2\frac{Q^{(m)}_t(b)}{P^{(m)}_t(b)}\mathrm{d}t.
\end{equation}
\end{lemma}
\begin{proof} {(i) Applying the It\^o lemma to the product $M^{(m)}_t(z)=P^{(m)}_t(z)N^{(m)}_t$ and using~\eqref{eq:dPm=}, one obtains}
\begin{align}
\notag \mathrm{d}M^{(m)}_t(z)\ &=\ P^{(m)}_t(z)\mathrm{d}N^{(m)}_t+N_t^{(m)}\mathrm{d}P^{(m)}_t(z)+\mathrm{d}\langle P^{(m)}(z),N^{(m)}\rangle_t\\
\label{x:dMm1=} & =\ P^{(m)}_t(z)\biggl[\mathrm{d}N^{(m)}_t+2\pi m^2N_t^{(m)}\!\!\left[\int_{\Omega_t}P_t(w)P^{(m)}_t(w)dA(w)\right]\mathrm{d}t\biggr]\\
\label{x:dMm2=} &\,+\ Q^{(m)}_t(z)\bigl[N^{(m)}_t\mathrm{d}\xi_t+\mathrm{d}\langle\xi,N^{(m)}\rangle_t\bigr].
\end{align}

Recall (see Remark~\ref{rem:Pm-mart}) that the process~$\mathrm{d}M^{(m)}_{t\wedge\tau_r}(z)$ should be a martingale \emph{for each $z\in B_\Omega(b,\frac{1}{2}r)$} and it is obvious that the functions $P^{(m)}_{t\wedge\tau_r}(\cdot),\ Q^{(m)}_{t\wedge\tau_r}(\cdot)$ are linearly independent. Thus, the only possibility is that
\[
\text{both terms~\eqref{x:dMm1=} and~\eqref{x:dMm2=} are local martingales}
\]
(until the stopping time~$\tau_r$ which almost surely grows to infinity as $r\to 0$). The bounded variation (drift) part of~$N^{(m)}_t$ can be easily identified from~\eqref{x:dMm1=}. To identify the martingale part, recall (see~\eqref{eq:defNm} and~\eqref{eq:Pm/P(b)def}) that
\[
N^{(m)}_t\ =\ \frac{P_t(b)}{P^{(m)}_t(b)}\ =\ \biggl[1-m^2\int_{\Omega_t}P_t(w)\frac{G^{(m)}_t(w,b)}{P_t(b)}dA(w)\biggr]^{-1},
\]
where, as usual, we use the shorthand notation
\[
\frac{G^{(m)}_t(w,b)}{P_t(b)}\ :=\ \lim_{w\to b}\frac{G^{(m)}_t(w,z)}{P_t(z)}.
\]
As~$Q_t(b)/P_t(b)=0$ (see~\eqref{eq:Q(b)/P(b)=0}), the massive Hadamard formula~\eqref{eq:mHadamard} gives
\[
\mathrm{d}\frac{G^{(m)}_t(w,b)}{P_t(b)}\ =\ -X_t(w)\mathrm{d}t,\quad \text{where}\quad X_t(w):=2\pi P^{(m)}_t(w)(N^{(m)}_t)^{-1}\le 2\pi P_t(w).
\]
Therefore,
\[
\mathrm{d}\biggl[P_t(w)\frac{G_t^{(m)}(w,b)}{P_t(b)}\biggr]\ =\ Q_t(w)\frac{G^{(m)}_t(w,b)}{P_t(b)}\mathrm{d}\xi_t-P_t(w)X_t(w)\mathrm{d}t.
\]
It follows from~\eqref{eq:QGbound} that~$|Q_t(w)|{G^{(m)}_t(w,b)}/{P_t(b)}\le  CP_t(w)$, thus one can apply the stochastic Fubini theorem
as in the proof of Lemma~\ref{lem:fubini-ok} and conclude that
\[
\mathrm{d}\frac{1}{N^{(m)}_t}=-m^2\left[\int_{\Omega_t}Q_t(w)\frac{G^{(m)}_t(w,b)}{P_t(b)}dA(w)\right] \mathrm{d}\xi_t+m^2\left[\int_{\Omega_t}P_t(w)X_t(w)dA(w)\right]\mathrm{d}t.
\]
{Since, due to~\eqref{eq:Qmdef} and~\eqref{eq:Q(b)/P(b)=0}, we have}
\[
-m^2\int_{\Omega_t}Q_t(w)\frac{G^{(m)}_t(w,b)}{P_t(b)}dA(w)\ =\ 
\frac{Q^{(m)}_t(b)}{P_t(b)}\ =\ \frac{Q^{(m)}_t(b)}{P^{(m)}_t(b)}(N^{(m)}_t)^{-1},
\]
{the martingale part of the process~$N_t^{(m)}$ coincides with that in the formula~\eqref{eq:dN=}. Recall that the bounded variation part of $N_t^{(m)}$ is already identified by~\eqref{x:dMm1=}.}

\smallskip

\noindent {(ii) Recall that we know that the martingale part of the process~$\xi_t$ is given by~$\sqrt{2}B_t$. Therefore, we can use the fact that~\eqref{x:dMm2=} is a local martingale together with the identification of the (martingale part of the) process $N^{(m)}_t$ made above in order to identify the drift $\lambda_t\mathrm{d}t$ of the process~$\xi_t$. This gives the required formula~\eqref{eq:dlambda=}.}
\end{proof}

\subsection{Proof of Theorem~\ref{theorem}} \label{sub:prfThm11}
For convenience of the reader, we now {briefly} summarize the proof of Theorem~\ref{theorem}, {which consists of the following two parts:}
\begin{enumerate}\renewcommand\theenumi{\roman{enumi}}
\item The results of Section~\ref{sect:density} imply that the Radon--Nikodym derivatives of massive LERW measures~$\mathbb P^{(m)}_{(\Od;a^\delta,b^\delta)}$ with respect to the classical ($m=0$) ones are uniformly bounded. Therefore, the discussion of tightness given in Section~\ref{sect:topologies} also applies to these measures. {Moreover,} as argued in Section~\ref{sect:density}, each subsequential limit~$\mathbb P^{\mathbb (m)}_{(\Omega;a,b)}$ of those is necessarily absolutely continuous with respect to the classical SLE(2) measure~$\mathbb P_{(\Omega;a,b)}$. This justifies the application of the Girsanov theorem {and implies that the driving term $\xi_t$ of the Loewner evolution~\eqref{eq:Loewner} is a semi-martingale}
    \[
    \mathrm{d}\xi_t=\sqrt{2}\mathrm{d}{B}_t+2\lambda_t\mathrm{d}t\quad \text{under}\ \ \mathbb{P}^{(m)}_{(\Omega;a,b)}.
    \]

\item Due to Remark~\ref{rem:Pm-mart}, the scaling limits of martingale observables~\eqref{eq:defMm} (provided by Proposition~\ref{prop:Mconv}, {which is} the main result of Section~\ref{sect:convergence}) are martingales under~$\mathbb P^{\mathbb (m)}_{(\Omega;a,b)}$. {As shown in Lemma \ref{lem:dNmt=},} this property is sufficient to identify the {drift term~$2\lambda_t\mathrm{d}t$} via brute force computations indicated in~\cite{Makarov-Smirnov} and a priori estimates from Sections~\ref{sect:estimates} and~\ref{sect:Hadamard}.
\end{enumerate}

\begin{remark} \label{rem:int-lambda2}
As mentioned in~\cite{Makarov-Smirnov}, the a priori (weak) uniqueness of a solution to the SDE $\mathrm{d}\xi_t = \sqrt{2}\mathrm{d}B_t+2\lambda_t\mathrm{d}t$ with~$\lambda_t:=Q^{(m)}_t(b)/P^{(m)}_t(b)$ follows from the fact that~$\int_0^{+\infty}\lambda_t^2dt\le\const(m,R)<\infty$ almost surely (which clearly implies the standard Novikov condition $\mathbb{E}[\exp(\frac{1}{2}\int_0^T \lambda_t^2 dt)]< \infty$ for all~$T>0$). Indeed,
\[
\lambda_t\ =\ -m^2\left[\int_{\Omega_t}Q_t(w)\frac{G^{(m)}_t(w,b)}{P_t(b)}dA(w)\right] N^{(m)}_t,
\]
the factor~$N^{(m)}_t=P_t(b)/P^{(m)}_t(b)$ is uniformly bounded due to~\eqref{eq:Pm/P>c} and hence
\[
\int_0^{+\infty}|\lambda_t|^2dt\ \le\ \const(m,R)\cdot\int_0^{+\infty}\biggl[\int_{\Omega_t}P_t(w)dA(w)\biggr]^2dt\ \le\ \const(m,R)\ <\ \infty
\]
due to the uniform estimate~\eqref{eq:QGbound} and the result of Corollary~\ref{cor:intPconv}(i).
\end{remark}

\pagebreak

\begin{remark} \label{rem:Nm-vs-Nb}
We conclude the paper by coming back to the parallel, already mentioned in Remark~\ref{rem:Nm-not-mart}, of the `massive/critical' setup discussed in this  paper and more standard `critical/critical' ones. Though the process~$N^{(m)}_t$ does \emph{not} coincide with the density~$(D^{(m)}_t)^{-1}$ and hence one cannot find~$\lambda_t$ directly from~\eqref{x:lambda}, only its martingale part plays a role in the identification of~$\xi_t$ via the martingale property of the process~$N^{(m)}_t\mathrm{d}\xi_t+\mathrm{d}\langle\xi,N^{(m)}\rangle_t$; { see~\eqref{x:dMm2=}.} This is the reason why the drift {term~$2\lambda_t\mathrm{d}t$} in Theorem~\ref{theorem} has exactly the same form as, e.g., in~\cite{zhan-04,izyurov-multiple,wu-hypergeomSLE,kemppainen-smirnov-iv}.
\end{remark}


\end{document}